\documentclass[a4paper,10pt]{amsart}
\usepackage[pdftex]{color,graphicx}
\usepackage{amssymb}
\usepackage{amsfonts}
\usepackage{esint}
\usepackage{relsize}
\usepackage{amsmath}
\usepackage{amsrefs}
\usepackage{epsfig}
\usepackage{amsthm}
\usepackage{color}
\usepackage[]{epsfig}
\usepackage[]{pstricks}
\usepackage[utf8]{inputenc}
\usepackage{multicol}
\usepackage{subfigure}
\newtheorem{theorem}{Theorem}[section]
\newtheorem{proposition}{Proposition}[section]
\newtheorem{lemma}{Lemma}[section]
\newtheorem{definition}{Definition}[section]
\newtheorem{example}{Example}[section]
\newtheorem{corollary}{Corollary}[section]

\newtheorem{remark}{Remark}[section]
\newcommand{\R}{\mathbb{R}}
\newcommand{\h}{\mathbb{H}}
\newcommand{\s}{\mathbb{S}}

\newcommand{\ria}{\rightarrow}

\newcommand{\de}{\partial}

\newcommand{\om}{\omega}
\newcommand{\n}{\nabla}

\newcommand{\ran}{\rangle}
\newcommand{\lan}{\langle}
\newcommand{\ve}{\varepsilon}
\newcommand{\vp}{\varphi}

\DeclareMathOperator{\hess}{Hess}
\DeclareMathOperator{\ric}{Ric}
\DeclareMathOperator{\rad}{rad}
\DeclareMathOperator{\di}{div}
\DeclareMathOperator{\tr}{tr}
\DeclareMathOperator{\degree}{degree}
\DeclareMathOperator{\trace}{trace}

\DeclareMathOperator{\scal}{scal}
\DeclareMathOperator{\genus}{genus}

\numberwithin{equation}{section}
\title[Stability of CMC surfaces in warped products]{Stability of constant mean curvature surfaces in three dimensional warped product manifolds}
\author{Greg\'orio Silva Neto}
\date{March 06, 2019}

\address{Instituto de Matemática\\
Universidade Federal de Alagoas\\ 
Macei\'o, AL, 57072-900, Brasil\\}
\email{gregorio@im.ufal.br}

\begin{document}


\footnotetext{G. Silva Neto was partially supported by the National Council for Scientific and Technological Development - CNPq of Brazil.}

\begin{abstract}

In this paper we prove that stable, compact without boundary, oriented, nonzero constant mean curvature surfaces in the de Sitter-Schwarzschild and Reissner-Nordstrom manifolds are the slices, provided its mean curvature satisfies some positive lower bound. More generally, we prove that stable, compact without boundary, oriented nonzero constant mean curvature surfaces in a large class of three dimensional warped product manifolds are embedded topological spheres, provided the mean curvature satisfies a positive lower bound depending only on the ambient curvatures. We conclude the paper proving that a stable, compact without boundary, nonzero constant mean curvature surface in a general Riemannian is a topological sphere provided its mean curvature has a lower bound depending only on the scalar curvature of the ambient space and the squared norm of the mean curvature vector field of the immersion of the ambient space in some Euclidean space.
\end{abstract}
\maketitle

\section{Introduction}
\label{intro}
In 1984, see \cite{BdC}, Barbosa and do Carmo introduced the notion of stability of compact hypersurfaces with nonzero constant mean curvature in the Euclidean space. They proved that the only compact, without boundary, hypersurfaces with nonzero constant mean curvature of the Euclidean space are the round spheres. Later, in 1988, in a joint work with Eschenburg, see \cite{BdCE}, Barbosa and do Carmo extended the notion of stability for hypersurfaces of a general Riemannian manifold and proved that the only compact, without boundary, nonzero constant mean curvature hypersurfaces of the Euclidean sphere and the hyperbolic space are the geodesic spheres. 

Briefly speaking, a compact, without boundary, nonzero constant mean curvature surface $\Sigma$ of a three dimensional Riemannian manifold is stable if, and only if, it is a local minimum of the area functional under all normal variations which preserve volume. This means that $J''(0)(f)\geq 0$ for all smooth function $f:\Sigma\rightarrow\R$ satisfying 
\[
\int_\Sigma f d\Sigma =0,
\]
where
\begin{equation}\label{jacobi}
\begin{split}
J''(0)(f) &= -\int_\Sigma \left[f\Delta_\Sigma f + (\ric_M(N,N) + \|A\|^2)f^2\right] d\Sigma\\    
          &= \int_\Sigma \left[|\n_\Sigma f|^2 - (\ric_M(N,N) + \|A\|^2)f^2\right]d\Sigma. 
\end{split}
\end{equation}

Here, $\ric_M(N,N)$ denotes the Ricci tensor of $M$ in the direction of the unitary vector field $N,$ normal to $\Sigma,$ $\Delta_\Sigma f$ denotes the Laplacian of $f$ over $\Sigma,$ $\n_\Sigma f$ denotes the gradient of $f$ over $\Sigma,$ and $\|A\|^2$ denotes the squared norm of the second fundamental form of $\Sigma.$ We refer to \cite{BdCE} for a detailed discussion of the subject.

Since the metrics of the space forms of constant sectional curvature $c\in\R$ can be written in polar coordinates as $\langle\cdot,\cdot\rangle= dt^2 + h(t)^2d\om^2,$ where
\[
h(t)=t \ \mbox{for} \ \R^3, \ h(t) = \dfrac{1}{\sqrt{c}}\sin(\sqrt{c}t)\ \mbox{for} \ \s^3(c),\ h(t) =\dfrac{1}{\sqrt{-c}}\sinh(\sqrt{-c}t) \ \mbox{for}\ \h^3(c),
\]
and $d\om^2$ denotes the canonical metric of the two-dimensional round sphere $\s^2,$ then is natural to ask if we can classify the compact, without boundary, stable nonzero constant mean curvature surfaces in the more general class of three-dimensional Riemannian manifolds $M^3=I\times\s^2,$ where $I=(0,b)$ or $I=(0,\infty),$ with the metric 
\begin{equation}\label{warped}
\langle\cdot,\cdot\rangle =dt^2+h(t)^2d\om^2,    
\end{equation}
with a more general smooth function $h:I\rightarrow\R.$ With the metric (\ref{warped}), the product $M^3=I\times\s^2$ is called a warped product manifold. These manifolds were first introduced by Bishop and O' Neill in 1969, see \cite{B-ON}, and is having increasing importance due to its applications as model spaces in general relativity. Part of these applications comes from the metrics which are solutions of the Einstein equations, as the de Sitter-Schwarzschild metric and Reissner-Nordstrom metric, which we introduce below.

\begin{definition}[The de Sitter-Schwarzschild manifolds]{\normalfont Let $m>0$ and $c\in\R.$ Let 
\[
(s_0,s_1)=\{r>0 ; 1-mr^{-1}-cr^2>0\}.
\] 
If $c\leq 0,$ then $s_1=\infty.$ If $c>0,$ assume that $cm^2<4/27.$ The de Sitter-Schwarzschild manifold is defined by $M^3(c)=(s_0,s_1)\times \s^2$ endowed with the metric
\[
\langle\cdot,\cdot\rangle=\dfrac{1}{1-mr^{-1}-cr^2}dr^2 + r^2 d\om^2.
\]
In order to write the metric in the form (\ref{warped}), define $F:[s_0,s_1)\rightarrow \R$ by
\[
F'(r)=\dfrac{1}{\sqrt{1-mr^{-1}-cr^2}}, \ F(s_0)=0.
\]
Taking $t=F(r),$ we can write $\langle\cdot,\cdot\rangle=dt^2+h(t)^2d\om^2,$ where $h:[0,F(s_1))\rightarrow[s_0,s_1)$ denotes the inverse function of $F.$ The function $h(t)$ clearly satisfies
\begin{equation}\label{defi-SS}
h'(t)=\sqrt{1-mh(t)^{-1}-ch(t)^2},\ h(0)=s_0,\ \mbox{and}\ h'(0)=0.
\end{equation}
}
\end{definition}

\begin{definition}[The Reissner-Nordstrom manifolds]\label{ex-RN}

{\normalfont The Reissner-Nordstrom manifold is defined by $M^3=(s_0,\infty)\times\s^2,$ with the metric
\[
\langle\cdot,\cdot\rangle=\dfrac{1}{1-mr^{-1}+q^2r^{-2}}dr^2 + r^2 d\om^2,
\]
where $m>2q>0$ and $s_0=2q^2/(m-\sqrt{m^2-4q^2})$ is the larger of the two solutions of $1-mr^{-1}+q^2r^{-2}=0.$ In order to write the metric in the form (\ref{warped}), define $F:[s_0,\infty)\rightarrow \R$ by
\[
F'(r)=\dfrac{1}{\sqrt{1-mr^{-1}+q^2r^{-2}}}, \ F(s_0)=0.
\]
Taking $t=F(r),$ we can write $\langle\cdot,\cdot\rangle=dt^2+h(t)^2d\om^2,$ where $h:[0,\infty)\rightarrow[s_0,\infty)$ denotes the inverse function of $F.$ The function $h(t)$ clearly satisfies
\begin{equation}\label{defi-RN}
h'(t)=\sqrt{1-mh(t)^{-1}+q^2h(t)^{-2}},\ h(0)=s_0,\ \mbox{and}\ h'(0)=0.
\end{equation}
}
\end{definition}
\begin{remark}
{\normalfont
Since the warped product manifold is smooth at $t=0$ if, and only if, $h(0)=0, \ h'(0)=1,$ and all the even order derivatives are zero at $t=0$, i.e., $h^{(2k)}(0) = 0,\ k > 0,$ see \cite{petersen}, Proposition 1, p. 13, we can see the de Sitter-Schwarzschild manifolds and the Reissner-Nordstrom manifolds are singular at $t=0.$
}
\end{remark}

In \cite{Brendle}, Brendle proved that the only compact, embedded, nonzero constant mean curvature hypersurfaces of a wide class of $n$-dimensional warped product manifolds, including the de Sitter-Schwarzschild and Reissner-Nordstrom manifolds are the slices $\{r_0\}\times\s^{n-1}.$ This inspire us to ask if we can replace the assumption of embeddedness by stability and obtain the same kind of result. This is reinforced by the fact that slices are stable in these spaces, see Proposition \ref{slice}, p. \pageref{slice}. For dimension 2, this is the subject of the next theorems.

\begin{theorem}\label{Stab-SS}
Let $\Sigma$ be a compact, without boundary, stable, constant mean curvature $H\neq 0$ surface of the de Sitter-Schwarzschild manifold. If $ \Sigma\subset [r_0,s_1)\times\s^2\subset M^3,$ $r_0\in(s_0,s_1),$ and 
\[
H^2\geq \dfrac{m}{2r_0^3} - c
\]
then $\Sigma$ is a slice.
\end{theorem}

\begin{theorem}\label{Stab-RN}
Let $\Sigma$ be a compact, without boundary, stable, constant mean curvature $H\neq 0$ surface of the Reissner-Nordstrom manifold. If $\Sigma\subset [r_0,\infty)\times\s^2\subset M^3,$ $r_0\in(s_0,\infty),$ $2q\leq\sqrt{15}m/4,$ and
\[
H^2\geq \dfrac{1}{2r_0^3}\left(m - \dfrac{2q^2}{r_0}\right),
\]
then $\Sigma$ is a slice.
\end{theorem}

\begin{remark}
{\normalfont
The slices $\{r\}\times\s^2$ satisfy the hypothesis of Theorem \ref{Stab-SS} if $r\geq 3m/2$ and satisfy the hypothesis of Theorem \ref{Stab-RN} if $r\geq (3m+\sqrt{9m^2-32q^2})/4$. In fact the slices satisfy the hypothesis of Theorems \ref{Stab-SS} and \ref{Stab-RN} when the mean curvature $H(r)$ of the slice $\{r\}\times\s^2$ is a decreasing function of $r.$ We prove this fact in a more general setting in Remark \ref{slice-2}, p. \pageref{slice-2} (Fig. \ref{fig-1}).
\begin{figure}[h]
\begin{center}
\includegraphics[scale=0.15]{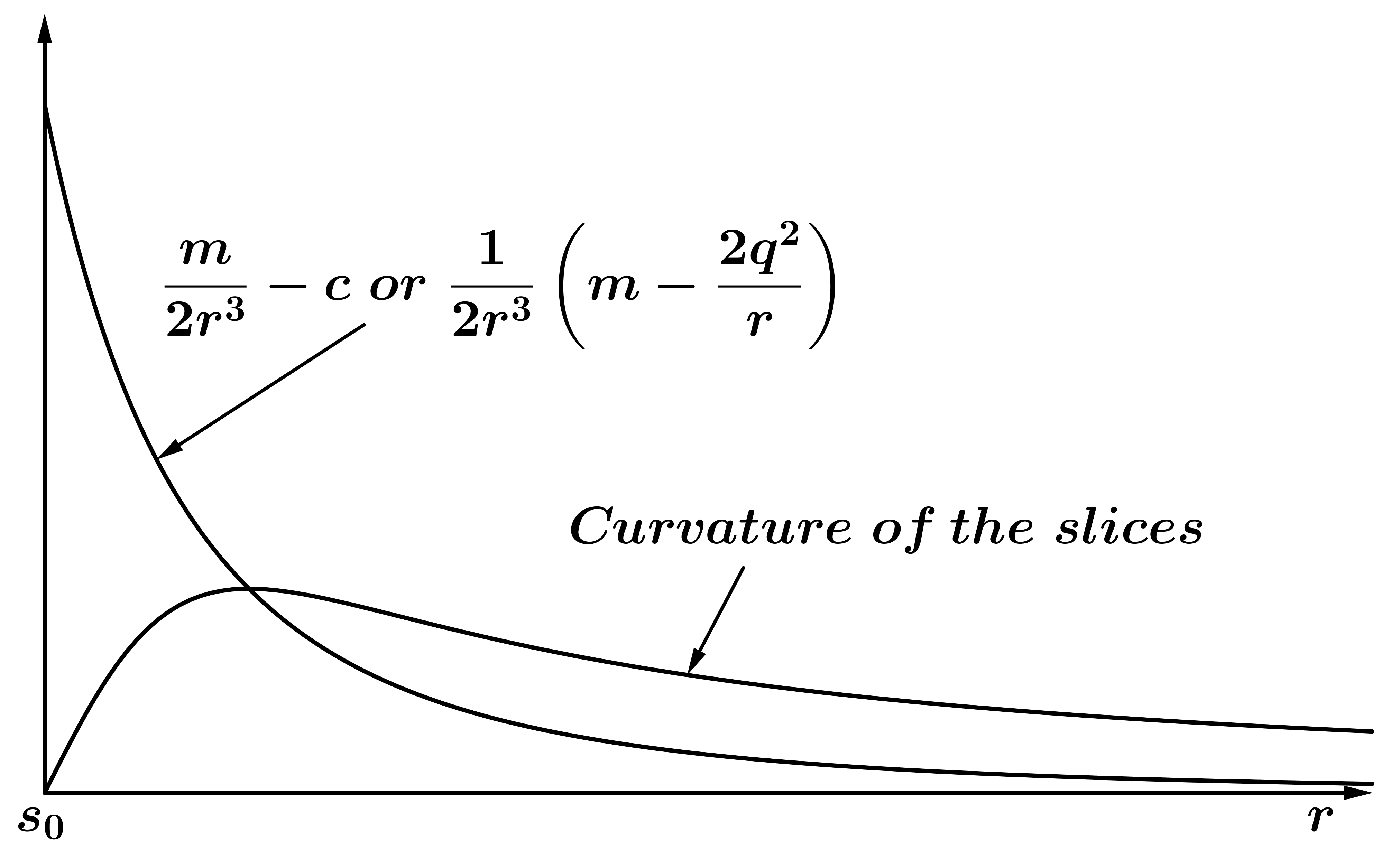}
\end{center}
\caption{Representation of the curvature of the slices compared with the hypothesis of Theorems \ref{Stab-SS} and \ref{Stab-RN}}
\label{fig-1}
\end{figure}
}
\end{remark}

\begin{remark}
{\normalfont
Since the de Sitter-Schwarzschild manifolds become the space forms by taking $m\rightarrow 0,$ from Theorem \ref{Stab-SS} we obtain, for dimension $2,$ the results of Barbosa, do Carmo and Eschenburg in $\R^3$ and the open hemisphere $\s^3_+,$ for every $H>0,$ and in $\h^3$ for $H\geq1.$
}
\end{remark}

In a more general setting, the warped product manifold $M^3=I\times\s^2$ with metric (\ref{warped}) has two different sectional curvatures, called tangential and radial curvatures, respectively:
\[
K_{\tan}(t)= K(X,Y) = \dfrac{1-h'(t)^2}{h(t)^2} \ \mbox{and} \ K_{\rad}(t)=K(X,\n t)=-\dfrac{h''(t)}{h(t)},
\]
where $\n t$ is the radial direction in polar coordinates and $X,Y\in TM$ satisfy $X,Y\perp \n t.$ 

In recent years, immersions in warped product manifolds have been extensively studied, see for example, \cite{aledo}, \cite{gimeno}, \cite{GIR}, \cite{sal-sal}, \cite{bessa}, \cite{Brendle}, \cite{BCL}, \cite{AIR}, \cite{AD-2}, \cite{AD-1} and \cite{montiel-1}.

In 2013, Brendle, see \cite{Brendle}, proved that, if $M^n$ is a $n$-dimensional warped product manifold whose sectional curvatures satisfy the inequality
\begin{equation}\label{brendle-c}
\dfrac{d K_{\rad}(t)}{dt}\leq (n-2)\frac{h'(t)}{h(t)}(K_{\tan}(t) - K_{\rad}(t)),
\end{equation}
then every compact, without boundary, embedded, constant mean curvature $H\neq 0$ hypersurface of $M$ is umbilic. Under some additional conditions on $M$, he proved that these surfaces must be a slice or a geodesic sphere. Condition (\ref{brendle-c}) seems to be necessary since, when the condition fails, there are small nonzero constant mean curvature spheres which are not umbilic (see Theorem 1.5, p. 250 of \cite{Brendle}).

On the other hand, as we can see in Proposition \ref{slice}, p. \pageref{slice}, the slices are stable if and only if $K_{\tan}(t)\geq K_{\rad}(t).$ Thus, the slices are not the natural candidates for every situation. Despite we do not know who are the natural candidates for every situation, in the next theorems we can prove that the compact, without boundary, stable constant mean curvature $H\neq0$ surfaces of a wide class of warped product manifolds, if exists, are spheres in a certain way. In the case when $K_{\tan}(t)>0,$ which means that $M^3$ can be immersed as hypersurface of revolution of $\R^4,$ we can prove:

\begin{theorem}\label{theo-warped-1}
Let $M^3=I\times\s^2$ be a warped product manifold, with metric $\langle\cdot,\cdot\rangle=dt^2 + h(t)^2d\om^2,$ such that, for every $t\in I,$ $K_{\tan}(t)>0$ and
\[
0\leq K_{\rad}(t) \leq (2+\sqrt{5})K_{\tan}(t).
\] 
If $\Sigma$ is a compact, without boundary, stable, constant mean curvature $H\neq0$ surface of $M^3,$ then $\genus(\Sigma)=0$ and $\Sigma$ is embedded.
\end{theorem}

\begin{remark}
{\normalfont 
As a particular cases of Theorem \ref{theo-warped-1} we obtain the results of Souam, see \cite{souam}, in $\s^2\times\R$ and the result of Barbosa, do Carmo and Eschenburg, see \cite{BdC}, for $\s^3.$ In fact in $\s^2\times\R$ we have $K_{\tan}(t)=1$ and $K_{\rad}(t)=0,$ and in $\s^3$ we have $K_{\tan}(t)=K_{\rad}(t)=1.$ The results then follow by using the Abresch-Rosenberg Hopf type theorem, see \cite{AR}, for $\s^2\times\R$ and the classical Hopf theorem for $\s^3.$
}
\end{remark}

If $M^3$ is a warped product manifold whose sectional curvatures $K_{\tan}(t)$ and $K_{\rad}(t)$ do not satisfy the hypothesis of Theorem \ref{theo-warped-1}, we can obtain the same conclusion of such theorem if the mean curvature of the stable surface satisfies some lower bound depending only on $M^3:$

\begin{theorem}\label{theo-warped-2}
Let $M^3=I\times\s^2$ be a warped product manifold with metric $\langle\cdot,\cdot\rangle=dt^2 + h(t)^2d\om^2$ such that, for every $t\in I,$ 
\[
K_{\rad}(t)<0<K_{\tan}(t) \ \ \mbox{or} \ \ K_{\rad}(t)\geq (2+\sqrt{5})K_{\tan}(t)>0.
\]
There exists a constant $c_0=c_0(M)>0,$ depending only on $M,$ such that, if $\Sigma$ is a compact, without boundary, stable, constant mean curvature surface $H\neq0$ of $M^3,$ and
\[
H^2>c_0,
\] 
then $\genus(\Sigma)=0$ and $\Sigma$ is embedded.
\end{theorem}

By the embeddedness of the stable surfaces proved in Theorems \ref{theo-warped-1} and \ref{theo-warped-2} and by using the results of Brendle in \cite{Brendle}, we obtain the next corollary:

\begin{corollary}
Let $M^3=I\times\s^2$ be a warped product manifold with metric $\langle\cdot,\cdot\rangle=dt^2 + h(t)^2d\om^2$ such that
\[
\dfrac{d K_{\rad}(t)}{dt}\leq \frac{h'(t)}{h(t)}(K_{\tan}(t) - K_{\rad}(t)).
\]
There exists a constant $c_0=c_0(M)\geq0,$ depending only on $M,$ such that, if $\Sigma$ is a compact, without boundary, stable, constant mean curvature surface $H\neq0$ of $M^3,$ and $H^2 >c_0,$
then $\Sigma$ is umbilic. Moreover,
\begin{itemize}
\item[i)] if $h(0)=0,$ $h''(0)>0$ $h'(t)>0$ for all $t\in I,$ and $K_{\tan}(t) > K_{\rad}(t),$ then $\Sigma$ is a slice;

\item[ii)] if $h(t)=t\varphi(t^2),$ where $\varphi:(0,\sqrt{b})\rightarrow\R$ is a smooth function satisfying $\varphi(0)=1,$ $h'(t)>0$ for all $t\in I$ and $K_{\tan}(t)\neq K_{\rad}(t),$ then $\Sigma$ is a geodesic sphere centered in the origin.
\end{itemize}
\end{corollary} 

\begin{remark}
{\normalfont 
As far as we know, there are few results about stability of surfaces or hypersurfaces in warped product manifolds. We can cite, as examples, \cite{montiel-1} for stability of compact, without boundary, constant mean curvature hypersurfaces, and \cite{aledo} and \cite{BCL} for stability of minimal submanifolds in warped product manifolds.
}
\end{remark}

\begin{remark}
{\normalfont 
As pointed out in the appendix of \cite{BdCE}, the stability problem is closed related to the isoperimetric problem, i.e., to find the surfaces with least area bounding a domain of given volume. Such surfaces are stable if they are smooth. For the space forms, the geodesic spheres are the solutions of the isoperimetric problem. In \cite{B-M}, Bray and Morgan proved that the slices are uniquely minimizing among all the surfaces enclosing the horizon $\{s_0\}\times\s^2$ for a class of warped manifolds, including the Schwarzschild manifold. The same result is true, see \cite{corvino}, for a class of warped product manifolds which includes the de Sitter-Schwarzschild manifold and Reissner-Nordstrom manifold. On the other hand, in \cite{ritore}, Ritor\'e has constructed examples of warped product surfaces such that there are no solutions of the isoperimetric problem for any volume. 
}
\end{remark}

We conclude the paper with a theorem for compact, without boundary, stable, nonzero constant mean curvature surfaces in general three-dimensional Riemannian manifolds isometrically immersed in some $\R^n.$% which we know the second fundamental form of its immersion in some $\R^n.$

\begin{theorem}\label{theo-stab-general}
Let $M^3$ be a three-dimensional Riemannian manifold and $\Sigma$ be a compact, without boundary, stable, constant mean curvature $H\neq0$ surface of $M^3.$ Let $\scal_M$ be the normalized scalar curvature of $M$ and $\mathcal{H}$ be the normalized mean curvature vector of $M^3$ in $\R^n,\ n\geq 4.$ If 
\[
H^2>-3\inf_{\Sigma}\left[\scal_M - \frac{3}{4}\|\mathcal{H}\|^2\right],
\]
then $\genus(\Sigma)=0.$
\end{theorem}
In particular, we have

\begin{corollary}\label{stab-minimal}
Let $M^3$ be a three-dimensional Riemannian manifold which can be minimally immersed in some $\R^n$ and $\Sigma$ be a compact, without boundary, stable, constant mean curvature $H\neq0$ surface of $M^3.$ If
\[
H^2>-3\inf_\Sigma \scal_M,
\]
then $\genus(\Sigma)=0.$
\end{corollary}

\begin{remark}\label{rem-Frensel}
{\normalfont The results above can be compared, for example, with the following result of K. Frensel, see \cite{frensel}:
\emph{
Let $\Sigma$ be a compact, without boundary, stable, constant mean curvature $H\neq0$ surface of a three-dimensional manifold $M^3.$ If
\[
H^2> - \dfrac{1}{2}\inf_{\Sigma} \ric_M
\]
where $\ric_M$ is the Ricci curvature of $M^3,$ then $\genus(\Sigma)\leq 3.$
}
}
\end{remark}

This paper is organized as follows. In Section 2 we prove Theorems \ref{Stab-SS} and {\ref{Stab-RN}}. In Section 3, we discuss the relation between harmonic vector fields and stability. In Section 4 we prove Theorems \ref{theo-warped-1} and Theorem \ref{theo-warped-2}. We conclude the paper in Section 5 proving Theorem \ref{theo-stab-general}.

\emph{Acknowledgements:} The author would like to thank Hil\'ario Alencar by helpful conversations during the preparation of this paper and to the anonymous referee by the useful observations.

\section{Proof of Theorems \ref{Stab-SS} and \ref{Stab-RN}}

We start with the following lemma, whose proof can be found in \cite{oneill}, p. 210, Proposition 42:

\begin{lemma}
Let $M^3=I\times \s^2$ be a warped product manifold with the warped metric $\langle\cdot,\cdot\rangle = dt^2 + h(t)^2d\om^2,$ where $d\om^2$ is the canonical metric of the round sphere $\s^2$ and $h:I\rightarrow\R$ is the smooth warping function. Denote by $\overline{R}$ the curvature tensor of $(M^3,\langle\cdot,\cdot\rangle).$ Then, for $U,V,W\in T\s^2$ (i.e., $U,V,W\perp \partial_t$, where $\partial_t$ is the dual vector field of $dt$),
\begin{itemize}

\item[i)] $\overline{R}(V,\partial_t)\partial_t = \dfrac{\hess h(\de_t,\de_t)}{h(t)}V = \dfrac{h''(t)}{h(t)}V = -K_{\rad}(t)V;$

\item[ii)] $\overline{R}(V,W)\de_t =0;$

\item[iii)] $\overline{R}(\de_t,V)W = \dfrac{\langle V,W\rangle}{h(t)}\overline{\n}_{\de_t}\overline{\n} h(t) = \langle V,W\rangle\dfrac{h''(t)}{h(t)}\de_t = -K_{\rad}(t)\langle V,W\rangle\de_t;$

\item[iv)] 
$\begin{array}{rl}
\overline{R}(U,V)W &= \dfrac{1}{h(t)^2}R^{\s^2}(U,V)W - \dfrac{\|\overline{\n} h(t)\|^2}{h(t)^2}[\langle U,W \rangle V - \langle V,W\rangle U]\\ 
&=K_{\tan}(t)[\langle U,W \rangle V - \langle V,W\rangle U],\\
\end{array}
$
\end{itemize}
where $\overline{\n}$ is the connection of $M^3$ and $R^{\s^2}$ is the curvature tensor of $\s^2.$

\end{lemma}

In the next proposition we will state a more suitable expression for the curvature tensor $\overline{R}.$

\begin{proposition}\label{tensor-curv}
Let $M^3=I\times \s^2$ be a warped product manifold with the warped metric $\langle\cdot,\cdot\rangle = dt^2 + h(t)^2d\om^2,$ where $d\om^2$ is the canonical metric of the round sphere $\s^2$ and $h:I\rightarrow\R$ is the smooth warping function. Denote by $\overline{R}$ the curvature tensor of $(M^3,\langle\cdot,\cdot\rangle).$ Then for $X,Y,Z\in TM,$
\[
\begin{split}
\overline{R}(X,Y)Z &= K_{\tan}(t)(\langle X,Z\rangle Y - \langle Y,Z\rangle X)\\
&\qquad - (K_{\tan}(t) - K_{\rad}(t))\left\langle\langle X,Z\rangle Y - \langle Y,Z\rangle X,\de_t\right\rangle \de_t\\
&\qquad - (K_{\tan}(t) - K_{\rad}(t))\langle Z,\de_t\rangle[\langle X,\de_t\rangle Y - \langle Y,\de_t\rangle X].
\end{split}
\]
\end{proposition}
\begin{proof}
Let $X=X_0 + a\de_t,$ $Y=Y_0 +  b\de_t$ and $Z=Z_0 + c\de_t,$ where $X_0,Y_0,Z_0\perp \de_t$ and $a=\langle X,\de_t\rangle,$ $b=\langle Y,\de_t\rangle$ and $c=\langle Z,\de_t\rangle.$ Then
\[
\begin{split}
\overline{R}(X,Y)Z &= \overline{R}(X_0 + a\de_t,Y_0 + b\de_t)(Z_0 + c\de_t)\\
                   &= \overline{R}(X_0,Y_0)Z_0 + c \overline{R}(X_0,Y_0)\de_t + b \overline{R} (X_0,\de_t)Z_0 + bc\overline{R}(X_0,\de_t)\de_t\\
                   &\quad + a\overline{R}(\de_t,Y_0)Z_0 + ac\overline{R}(\de_t,Y_0)\de_t\\
                   &= K_{\tan}(t)(\langle X_0,Z_0\rangle Y_0 - \langle Y_0,Z_0\rangle X_0)\\
                   &\quad  - K_{\rad}(t)(a\langle Y_0,Z_0\rangle - b\langle X_0,Z_0\rangle)\de_t - K_{\rad}(t)(bc X_0 - ac Y_0).\\  
\end{split}
\]
On the other hand, since
\[
\langle X_0,Y_0\rangle = \langle X,Y\rangle  - \langle X,\de_t\rangle\langle Y,\de_t\rangle,
\]
and analogously for $\langle X_0,Z_0\rangle$ and $\langle Y_0,Z_0\rangle,$ we have
\[
\begin{split}
\langle X_0,Z_0\rangle Y_0 - \langle Y_0,Z_0\rangle X_0 &= (\langle X,Z\rangle  - \langle X,\de_t\rangle\langle Z,\de_t\rangle)(Y-\langle Y,\de_t\rangle\de_t) \\
                                            &\qquad -(\langle Y,Z\rangle  - \langle Y,\de_t\rangle\langle Z,\de_t\rangle)(X-\langle X,\de_t\rangle\de_t)\\
                                            &=\langle X,Z\rangle Y - \langle X,Z\rangle\langle Y,\de_t\rangle\de_t  - \langle X,\de_t\rangle\langle Z,\de_t\rangle Y\\
                                            &\qquad - \langle Y,Z\rangle X + \langle Y,Z\rangle\langle X,\de_t\rangle\de_t + \langle Y,\de_t\rangle\langle Z,\de_t\rangle X\\ 
                                            &= (\langle X,Z\rangle Y - \langle Y,Z\rangle X) - \langle \langle X,Z\rangle Y - \langle Y,Z\rangle X, \de_t\rangle\de_t\\
                                            &\quad - \langle Z,\de_t\rangle(\langle X,\de_t\rangle Y - \langle Y,\de_t\rangle X),\\
a\langle Y_0,Z_0\rangle - b \langle X_0,Z_0\rangle& = \langle X,\de_t\rangle(\langle Y,Z\rangle - \langle Y,\de_t\rangle\langle Z,\de_t\rangle)\\
                                      & \qquad - \langle Y,\de_t\rangle (\langle X,Z\rangle - \langle X,\de_t\rangle \langle Z,\de_t\rangle)\\
                                      & =- \langle \langle X,Z\rangle Y - \langle Y,Z\rangle X, \de_t\rangle,\\ 
\end{split}
\]
and
\[
\begin{split}
bcX_0 - acY_0 & = \langle Y,\de_t\rangle\langle Z,\de_t\rangle (X - \langle X,\de_t\rangle\de_t) - \langle X,\de_t\rangle\langle Z,\de_t\rangle(Y-\langle Y,\de_t\rangle\de_t)\\
              & = \langle Z,\de_t\rangle(\langle Y,\de_t\rangle X - \langle X,\de_t\rangle Y).
\end{split}
\]
The result then follows.
 
\end{proof}

\begin{corollary}\label{ricci-N}
Let $M^3=I\times \s^2$ be a warped product manifold with the warped metric $\langle\cdot,\cdot\rangle = dt^2 + h(t)^2d\om^2,$ where $d\om^2$ is the canonical metric of the round sphere $\s^2$ and $h:I\rightarrow\R$ is the smooth warping function. Let $\Sigma$ be a surface of $M^3$ with unit normal vector field $N.$ Then
\begin{equation}\label{normal-ricci}
\ric_M(N,N)=2K_{\tan}(t) + (K_{\rad}(t) - K_{\tan}(t))(1+\nu^2)
\end{equation}
or yet
\begin{equation}\label{normal-ricci-2}
\ric_M(N,N)= 2K_{\rad}(t) + (K_{\tan}(t)-K_{\rad}(t))(1-\nu^2),
\end{equation}
where $\nu=\langle N,\partial_t\rangle.$
\end{corollary}
\begin{proof}
Let $\{\bar{e}_1,\bar{e}_2,\bar{e}_3\}$ be an orthonormal basis of $M^3.$ By using Proposition \ref{tensor-curv}, we have
\[
\begin{split}
\overline{R}(N,\bar{e}_i)N &= K_{\tan}(t)(\bar{e}_i - \langle\bar{e}_i,N\rangle N) - (K_{\tan}(t)-K_{\rad}(t))\langle \bar{e}_i-\langle\bar{e}_i,N\rangle N,\partial_t\rangle\partial_t\\
&\qquad - (K_{\tan}(t) - K_{\rad}(t))\nu(\nu\bar{e}_i - \langle \bar{e}_i,\partial_t\rangle N).
\end{split}
\]
This implies
\[
\begin{split}
\ric_M(N,N)&=\sum_{i=1}^3\langle\overline{R}(N,\bar{e}_i)N,\bar{e}_i\rangle\\
& = 2K_{\tan}(t) - (K_{\tan}(t) - K_{\rad}(t))(1-\nu^2) - 2(K_{\tan}(t) - K_{\rad}(t))\nu^2\\
&=2K_{\tan}(t) - (K_{\tan}(t) - K_{\rad}(t))(1+\nu^2)\\
&=2K_{\tan}(t) + (K_{\rad}(t) - K_{\tan}(t))(1+\nu^2),\\
\end{split}
\]
which gives Equation (\ref{normal-ricci}). Rearranging the terms, we have
\[
\begin{split}
\ric_M(N,N) &= 2K_{\tan}(t) + (K_{\rad}(t) - K_{\tan}(t))(1+\nu^2)\\
            &= K_{\tan}(t) + K_{\rad}(t) + (K_{\rad}(t) - K_{\tan}(t))\nu^2\\
            &= 2K_{\rad}(t) + (K_{\tan}(t) - K_{\rad}(t)) - (K_{\tan}(t) - K_{\rad}(t))\nu^2\\
            &= 2K_{\rad}(t) + (K_{\tan}(t) - K_{\rad}(t))(1-\nu^2)
\end{split}
\]
which gives Equation (\ref{normal-ricci-2}).
 
\end{proof}
\begin{remark}
{\normalfont 
We can also obtain Equations (\ref{normal-ricci}) and (\ref{normal-ricci-2}) from the expression for the Ricci curvature found in \cite{Brendle}, pp. 248-249., by taking there $n=3,$ $N=\s^2,$ and $\rho=1.$
}
\end{remark}

We will also need the following result, whose proof can be found in \cite{souam}, p. 2850, Proposition 3.1. %See also \cite{ros-1}, Theorem 5, p.298.

\begin{proposition}%[Proposition 3.1 of \cite{souam}, p.2850. See also \cite{ros-1} Theorem 5. p.298]
\label{souam-0}
Let $M$ be a simply connected conformally flat Riemannian three-dimensional manifold and $\Sigma$ be a compact, orientable surface, without boundary, immersed in $M.$ Denote by $K_s$ the sectional curvature of $M$ evaluated on the tangent plane to $\Sigma,$ by $H$ the mean curvature of $\Sigma$ and by $d\Sigma$ its area element. Then
\begin{equation}\label{souam-1}
\int_\Sigma \left(H^2 + K_s \right)d\Sigma\geq 4\pi
\end{equation}
and the equality holds if, and only if, $\Sigma$ is a totally umbilic sphere.
Furthermore, if $\Sigma$ is not embedded, then
\begin{equation}\label{souam-2}
\int_\Sigma \left(H^2 + K_s\right)d\Sigma \geq 8\pi.
\end{equation}
\end{proposition}

The proof of the next lemma is essentially in \cite{ros-1}. See also \cite{ros-2} and \cite{souam}. We give a proof here for the sake of completeness.
\begin{lemma}
Let $\Sigma$ be a compact, without boundary, stable, constant mean curvature $H\neq0$ surface of a simply connected conformally flat Riemannian three-dimensional manifold $M.$ Let $g=\genus(\Sigma).$ If $g=2k$ or $g=2k+1,$ then
\begin{equation}\label{ric-genus-1}
\int_\Sigma \left(2H^2 + \ric_M(N,N)\right) d\Sigma \leq 8\pi(1-k)
\end{equation}
or, if $\Sigma$ is not embedded,
\begin{equation}\label{ric-genus-2} 
\int_\Sigma \left(2H^2 + \ric_M(N,N)\right) d\Sigma \leq -8\pi k.
\end{equation}
\end{lemma}
\begin{proof}
By a result of T. Meis for $g\geq 2$ (see \cite{meis}, p.51, for the original proof, also the remark in the page 152 of \cite{Forster} for the mention of $g\geq2,$ and \cite{G-H}, p.261), and by using Theorem 10-21, p. 275 of \cite{Sp}, for smaller genus, there exists a meromorphic map $\phi:\Sigma\rightarrow\s^2\subset\R^3$ such that
\begin{equation}\label{deg}
\degree (\phi) \leq 1 + \left[\dfrac{g+1}{2}\right],
\end{equation}
where $[x]$ is the greatest integer less or equal to $x.$ Composing $\phi$ with a conformal diffeomorphism of $\s^2$ we can suppose that
\[
\int_\Sigma \phi d\Sigma =0.
\]
This implies, by using the second equation of (\ref{jacobi}), p. \pageref{jacobi}, and the Gauss-Bonnet theorem,
\[
\begin{split}
0&\leq \int_\Sigma\left[ |\n\phi|^2 - (\ric_M(N,N) + \|A\|^2)\right] d\Sigma\\
 &= \int_\Sigma\left[ |\n\phi|^2 - (4H^2 + 2K_s + \ric_M(N,N) - 2K)\right] d\Sigma\\
 &= 8\pi \degree(\phi) - \int_{\Sigma}(4H^2 + 2K_s + \ric_M(N,N))d\Sigma + 8\pi(1-g).
\end{split}
\]
since $\|\phi\|^2 =1.$ Using estimate (\ref{deg}), we obtain
\begin{equation}\label{ric-0}
\int_\Sigma \left(4H^2 + 2K_s + \ric_M(N,N)\right)d\Sigma \leq 8\pi\left(2-g + \left[\dfrac{g+1}{2}\right]\right).
\end{equation}
This implies, by using (\ref{souam-1}),
\begin{equation}\label{ric-1}
\int_\Sigma \left(2H^2 + \ric_M(N,N)\right)d\Sigma \leq 8\pi\left(1-g + \left[\dfrac{g+1}{2}\right]\right),
\end{equation}
or if $\Sigma$ is not embedded, by using (\ref{souam-2}),
\begin{equation}\label{ric-2}
\int_\Sigma \left(2H^2 + \ric_M(N,N)\right)d\Sigma \leq 8\pi\left(-g + \left[\dfrac{g+1}{2}\right]\right).
\end{equation}
The result then follows.
 
\end{proof}

As we discussed in the introduction, the slices are the natural candidates to be the compact, without boundary, stable, constant mean curvature $H\neq0$ surfaces in the warped product manifolds. However, in the next proposition we prove that the slices are stable if, and only if, $K_{\tan}(t)\geq K_{\rad}(t).$

\begin{proposition}\label{slice}
The slice $\Sigma=\{t\}\times \s^2$ is stable if, and only if, $K_{\tan}(t)\geq K_{\rad}(t).$
\end{proposition}

\begin{proof}
Let $f:\Sigma\rightarrow\R$ be an smooth function satisfying $\displaystyle{\int_\Sigma f d\Sigma =0}$. The first eigenvalue of the Laplacian in the slice $\Sigma^2=\{t\}\times \s^2$ is
\[
\lambda_1^\Delta = 2(H^2 + K_{\tan}(t)),
\]
where $H = \dfrac{h'(t)}{h(t)}$ is constant in the slice. Applying the data above in (\ref{jacobi}), using that the slice is umbilical, i.e., $\|A\|^2=2H^2,$ $\nu=\langle N,\partial_t\rangle=-1$ in the slices, and the Rayleigh characterization of $\lambda_1^\Delta$, we have
\[
\begin{split}
J''(0)f &= -\int_\Sigma f\Delta_\Sigma f d\Sigma - \int_{\Sigma} (\ric_M(N,N) + \|A\|^2)f^2d\Sigma\\
        &\geq \lambda_1^\Delta\int_\Sigma f^2 d\Sigma - 2K_{\rad}(t)\int_\Sigma f^2 d\Sigma - 2H^2\int_\Sigma f^2 d\Sigma\\
        &= (\lambda_1^\Delta - 2H^2 - 2K_{\rad}(t))\int_\Sigma f^2 d\Sigma \\
        &= 2(K_{\tan}(t) - K_{\rad}(t))\int_\Sigma f^2 d\Sigma.\\
\end{split}
\]
Thus, if $K_{\tan}(t) \geq K_{\rad}(t),$ then $J''(0)f\geq0$ and $\Sigma$ is stable. Conversely, if $K_{\tan}(t) < K_{\rad}(t),$ by taking $f$ as the first eigenfunction, we have
\[
J''(0)f = 2(K_{\tan}(t) - K_{\rad}(t))\int_\Sigma f^2 d\Sigma < 0.
\]
Therefore $\Sigma$ is unstable.
 
\end{proof}
\begin{remark}
{\normalfont
Proposition \ref{slice} holds for every dimension with the same proof, just adapting the dimension.
}
\end{remark}

Now we state the main step in the proof of Theorems \ref{Stab-SS} and \ref{Stab-RN}:

\begin{theorem}\label{Stab-Main}
Let $M^3=I\times\s^2$ be a warped product manifold with metric $\langle\cdot,\cdot\rangle=dt^2 + h(t)^2d\om^2,$ where $h:I\rightarrow\R$ is the smooth warping function. Let $\Sigma$ be a compact, without boundary, stable, constant mean curvature $H\neq0$ surface of $M^3.$ If one of the following conditions is satisfied
\begin{itemize}
\item[(i)] $K_{\tan}(t)\geq K_{\rad}(t)$ and
\[
H^2\geq\sup_\Sigma \{-K_{\rad}(t)\};
\]
\item[(ii)] $K_{\rad}(t)\geq K_{\tan}(t),$ and
\[
H^2\geq\sup_\Sigma\{-K_{\tan}(t)\};
\]
\end{itemize}
then $\genus(\Sigma)\leq 1$ and $\Sigma$ is embedded.
\end{theorem}
\begin{proof}%[Proof of Theorem \ref{Stab-Main}]
If $K_{\tan}(t)\geq K_{\rad}(t),$ then the hypothesis $H^2\geq \sup_\Sigma \{-K_{\rad}(t)\}$ of item (i) and Equation (\ref{normal-ricci-2}) of Corollary \ref{ricci-N}, p.\pageref{normal-ricci-2}, imply
\[
\begin{split}
2H^2 + \ric_M(N,N)&=2H^2 + 2K_{\rad}(t) + (K_{\tan}(t)-K_{\rad}(t))(1-\nu^2)\\
                  &\geq 2H^2 + 2K_{\rad}(t)\geq 0.\\                  
\end{split}
\]
Analogously, if $K_{\rad}(t)\geq K_{\tan}(t),$ the hypothesis $H^2\geq\sup_\Sigma\{-K_{\tan}(t)\}$ of item (ii) and Equation (\ref{normal-ricci}) of Corollary \ref{ricci-N}, p.\pageref{normal-ricci}, imply
\[
\begin{split}
2H^2 + \ric_M(N,N)&=2H^2 + 2K_{\tan}(t) + (K_{\rad}(t) - K_{\tan}(t))(1+\nu^2)\\
                  &\geq 2H^2 + 2K_{\tan}(t) \geq 0.\\ 
\end{split}
\]
Thus, by using (\ref{ric-genus-1}) we have, for both items (i) and (ii),
\[
\begin{split}
0&\leq \int_\Sigma \left(2H^2 + \ric_M(N,N)\right)d\Sigma \leq 8\pi(1-k).
\end{split}
\]
This implies that $k=0$ or $k=1.$ If $k=1,$ i.e., $\genus(\Sigma)=2$ or $\genus(\Sigma)=3,$ then
\[
\int_\Sigma \left(2H^2 + \ric_M(N,N)\right)d\Sigma\equiv 0
\]
and all the inequalities become equalities. This implies that inequality (\ref{ric-genus-1}), p. \pageref{ric-genus-1}, becomes an equality. Since in the proof of inequality (\ref{ric-genus-1}) we used the inequality (\ref{souam-1}), of Proposition \ref{souam-0}, p. \ref{souam-1}, we deduce that inequality (\ref{souam-1}) also becomes an equality, i.e.,
\[
\int_\Sigma \left(H^2 + K_s\right) d\Sigma \equiv 4\pi 
\] 
and thus, by Proposition \ref{souam-0}, $\Sigma$ is a totally umbilic sphere. This is a contradiction with the assumption that $\genus(\Sigma)=2$ or $\genus(\Sigma)=3.$ Thus $\genus(\Sigma)\leq 1.$ On the other hand, if $\Sigma$ is not embedded, then by using (\ref{ric-genus-2}),
\[
0\leq \int_\Sigma \left(2H^2 + \ric_M(N,N)\right) d\Sigma \leq 0.
\]
Once again, this proves that $\Sigma$ is an umbilical sphere. This is a contradiction since we are assuming that $\Sigma$ is not embedded. Thus $\Sigma$ is embedded. This proves items (i) and (ii) of Theorem \ref{Stab-Main}. %The conclusion of the proof of Theorem \ref{Stab-Main} comes from Theorem 1.1, p. 248 of \cite{Brendle}.
 
\end{proof}

\begin{remark}\label{slice-2}
{\normalfont
The slices satisfy the hypothesis of Theorem \ref{Stab-Main} (i) if, and only if,
\[
\dfrac{d}{dt}\left(\dfrac{h'(t)}{h(t)}\right)\leq 0,
\]
i.e., if, and only if, the mean curvature $H(t)=\frac{h'(t)}{h(t)}$ of the slice $\{t\}\times\s^2$ is a non-increasing function of $t.$ In fact, we need to prove that
\[
H(t)^2=\left(\dfrac{h'(t)}{h(t)}\right)^2\geq - K_{\rad}(t) = \dfrac{h''(t)}{h(t)}. 
\]
This is equivalent to
\[
h''(t)h(t) - h'(t)^2\leq 0.
\]
On the other hand
\[
\dfrac{d}{dt}\left(\dfrac{h'(t)}{h(t)}\right) = \dfrac{h''(t)h(t)-h'(t)^2}{h(t)^2}.
\]
The claim then follows.
}
\end{remark}

Now let us prove Theorems \ref{Stab-SS} and \ref{Stab-RN}.

\begin{proof}[Proof of Theorem \ref{Stab-SS}.]
Since
\[
K_{\tan}(t) = \dfrac{1-h'(t)^2}{h(t)^2} = \dfrac{m}{h(t)^3}+c \ \mbox{and} \ K_{\rad}(t) = -\dfrac{h''(t)}{h(t)} = -\dfrac{m}{2h(t)^3}+c
\]
we have $K_{\tan}(t) > K_{\rad}(t).$ Thus applying Theorem \ref{Stab-Main} (i), if $\Sigma\subset [r_0,s_1)\times\s^2\subset M^3,$ and
\[
H^2 \geq \sup_{\Sigma} \dfrac{m}{2h(t)^3}-c = \dfrac{m}{2r_0^3}-c,
\]
then $\genus(\Sigma)\leq 1$ and $\Sigma$ is embedded. From Corollary 1.2, p. 249 of \cite{Brendle}, the only compact, embedded, nonzero constant mean curvature surfaces of the de Sitter-Schwarzschild manifold are the slices.
 
\end{proof}

\begin{proof}[Proof of Theorem \ref{Stab-RN}.]
Since
\[
K_{\tan}(t) = \dfrac{1-h'(t)^2}{h(t)^2} = \dfrac{1}{2h(t)^3}\left(2m - \dfrac{2q^2}{h(t)}\right)
\]
and
\[
K_{\rad}(t) = -\dfrac{h''(t)}{h(t)} = -\dfrac{1}{2h(t)^3}\left(m - \dfrac{2q^2}{h(t)}\right)
\]
we have $K_{\tan}(t) > K_{\rad}(t).$ Thus applying Theorem \ref{Stab-Main} (i), if $\Sigma\subset [r_0,\infty)\times\s^2\subset M^3,$ and
\[
H^2 \geq \sup_{\Sigma} \dfrac{1}{2h(t)^3}\left(m - \dfrac{2q^2}{h(t)}\right) = \dfrac{1}{2r_0^3}\left(m - \dfrac{2q^2}{r_0}\right),
\]
then $\genus(\Sigma)\leq 1$ and $\Sigma$ is embedded. From Corollary 1.3, p. 249 of \cite{Brendle}, the only compact, embedded, nonzero constant mean curvature surfaces of the Reissner-Nordstrom manifold are the slices.
 
\end{proof}

\section{Stability and harmonic vector fields}

Let $\Sigma$ be an orientable Riemannian surface and denote by $H^1(\Sigma,\R)$ the space of  harmonic 1-forms on $\Sigma$. Recall that a 1-form $\omega$ on $\Sigma$ is harmonic if, and only if, it is closed, i.e., $(\n\om)(X,Y)=(\n\om)(Y,X)$ for all $X,Y\in T\Sigma,$ and co-closed, i.e., $(\n\om)(e_1,e_1)+(\n\om)(e_2,e_2)=0,$ where $\{e_1,e_2\}$ is an orthonormal frame of $T\Sigma.$ The following result will be useful and its proof can be found in \cite{petersen}, pp. 204-206:

\begin{lemma}\label{harm-field}
Let $\Sigma$ be a compact Riemannian surface, $\omega$ be a 1-form on $\Sigma$ and $X:\Sigma\rightarrow T\Sigma$ be its dual vector field, i.e., $\om(U)=\langle X,U\rangle,$ for all $U\in T\Sigma.$ Then $\omega$ is harmonic if, and only if,
\[
\di X =0 \ \mbox{and} \ \langle\n_Z X,Y \rangle = \langle\n_Y X,Z \rangle, \ \mbox{for all}\ Y,Z\in T\Sigma.
\]
In this case we call $X$ a harmonic vector field.
\end{lemma}

Let $X$ be a harmonic vector field on $\Sigma$ and $f:\Sigma\rightarrow\R$ be a smooth function. Since
\[
\di (fX) = f\di X + \langle X,\n f\rangle = \langle X,\n f\rangle,
\]
if $\Sigma$ is compact, without boundary, then by using divergence theorem we have
\begin{equation}\label{div-0}
\int_\Sigma \langle X,\n f\rangle d\Sigma =0.
\end{equation}
Thus, defining $u:\Sigma\rightarrow\R$ by $u=\langle X,\n f\rangle$ we have that $u$ is a mean zero function.

Since $\Sigma$ has dimension $2,$ we can consider the complex structure $J$ on $\Sigma$ which satisfies $J^2=-Id,$ where $Id$ is the identity map of $T\Sigma$ and
\[
\langle JY,Z\rangle = -\langle Y,JZ\rangle.
\]
The following lemma is well known and we give a proof here for the sake of completeness. It gives us another harmonic vector field:
\begin{lemma}
Let $\Sigma$ be a Riemannian surface and $X$ be a harmonic vector field. Then $JX$ is also a harmonic vector field, where $J$ is the complex structure of $\Sigma.$
\end{lemma}
\begin{proof}
Let $\{e_1,e_2\}$ be an orthonormal frame in $\Sigma$ which is geodesic at $p\in\Sigma.$ Since $J$ is the complex structure, we have $Je_1=e_2$ and $Je_2=-e_1.$ This implies that, at $p,$
\[
\begin{split}
\di JX &= \langle\n_{e_1}JX,e_1\rangle + \langle\n_{e_2}JX,e_2\rangle = e_1\langle JX,e_1\rangle + e_2\langle JX,e_2\rangle\\
&=-e_1\langle X,Je_1\rangle - e_2\langle X,Je_2\rangle = - e_1\langle X,e_2\rangle + e_2\langle X,e_1\rangle\\
&=-\langle\n_{e_1}X,e_2\rangle + \langle\n_{e_2}X,e_1\rangle=0.\\
\end{split}
\]
To prove that $\langle \n_Z JX, Y\rangle = \langle \n_Y JX, Z\rangle,$ we need only prove that $\langle \n_{e_1} JX, e_2\rangle=\langle \n_{e_2} JX, e_1\rangle$ and then use the linearity of the connection and the linearity of the inner product. Since $\di X =\langle \n_{e_1} X, e_1\rangle + \langle \n_{e_2} X, e_2\rangle=0,$ we have at $p,$
\[
\begin{split}
\langle \n_{e_1} JX, e_2\rangle&= e_1\langle JX, e_2\rangle = -e_1\langle X,Je_2\rangle = e_1\langle X,e_1\rangle\\
                         &=\langle \n_{e_1} X, e_1\rangle = - \langle \n_{e_2} X, e_2\rangle =-e_2\langle X,e_2\rangle\\
                         &= -e_2\langle X,Je_1\rangle=e_2\langle JX,e_1\rangle\\
                         &= \langle \n_{e_2} JX, e_1\rangle.\\
\end{split}
\]
Therefore, by using Lemma \ref{harm-field} we conclude that $JX$ is harmonic.
 
\end{proof}

This implies that the space of harmonic vector fields is even dimensional. In fact, by using the de Rham cohomology theory (see, for example, \cite{petersen}, p. 194), it can be proven that
\[
\dim H^1(\Sigma,\R) = 2\genus(\Sigma).
\]

The main strategy in the proof of Theorems \ref{theo-warped-1}, \ref{theo-warped-2}, and \ref{theo-stab-general}, is the following: We assume, by contradiction, that $\genus(\Sigma)\geq 1.$ This will give us two linearly independent harmonic vector fields (which we call $X$ and $JX$) and, by using the geometric assumptions of these theorems we will obtain a contradiction, concluding that $\genus(\Sigma)=0.$

First let us fix some notations. We will denote by $D$ the connection of $\R^n,$ $\overline{\n}$ the connection of $M^3$ and $\n$ the connection of $\Sigma.$ Denote also by $\overline{II}^\alpha,$ $\alpha=4,\ldots,n,$ the second fundamental forms of $M^3$ in $\R^n,$ and by $II$ the second fundamental form of $\Sigma$ in $M^3,$ with associated shape operator $A:T\Sigma\rightarrow T\Sigma.$ 

Let $E_1,E_2,\ldots,E_n$ the canonical basis of $\R^n,$ $u_i = \langle X,E_i\rangle,$ where $X$ is a harmonic field and $u_i^* = \langle JX,E_i\rangle$. Since $E_i$ are the gradient of the coordinate functions of $\R^n,$ $X$ and $JX$ are harmonic, and $\Sigma$ is compact without boundary, by equation (\ref{div-0}), p. \pageref{div-0}, we have
\[
\int_\Sigma u_i d\Sigma =0\ \mbox{and}\ \int_\Sigma u_i^* d\Sigma =0.
\] 
Given a smooth function $f:\Sigma\rightarrow\R,$ let
\[
Q(f,f) = f\Delta_\Sigma f + (\ric_M(N,N) + \|A\|^2)f^2,
\]
the integrand of the Jacobi operator (\ref{jacobi}), where here $N$ is the unitary normal vector field of $\Sigma$ in $M^3$. Let us denote by
\[
Q(X,X) = \sum_{i=1}^n Q(u_i,u_i) \ \mbox{and} \ Q(JX,JX) = \sum_{i=1}^n Q(u_i^*,u_i^*).
\]

The following two lemmas will be useful in the proof of the main proposition of this section.
\begin{lemma}\label{previous}
Let $\Sigma$ be a surface immersed in a three-dimensional Riemannian manifold $M^3$ and let $A:T\Sigma\rightarrow T\Sigma$ be the shape operator of $\Sigma$ with mean curvature $H.$ Then, for any vector field $X$ of $T\Sigma,$ we have
\[
\|AX\|^2 = 2H\langle AX,X\rangle  - K_e\|X\|^2,
\]
where $K_e = \det A$ is the extrinsic curvature of $\Sigma.$
\end{lemma}
\begin{proof}
Let $\{e_1,e_2\}$ be an orthonormal frame of eigenvectors of $A.$ We have $AX=\lambda_1\langle X,e_1\rangle e_1 + \lambda_2\langle X,e_2\rangle e_2$ and thus
\[
\begin{split}
2H\langle AX,X\rangle &= (\lambda_1+\lambda_2)\langle\lambda_1\langle X,e_1\rangle e_1 + \lambda_2\langle X,e_2\rangle e_2,\langle X,e_1\rangle e_1 + \langle X,e_2\rangle e_2 \rangle\\
&=(\lambda_1+\lambda_2)\lambda_1\langle X,e_1\rangle^2 + (\lambda_1+\lambda_2)\lambda_2\langle X,e_2\rangle^2\\
&=\lambda_1^2\langle X,e_1\rangle^2 + \lambda_2^2\langle X,e_2\rangle^2 + \lambda_1\lambda_2(\langle X,e_1\rangle^2 + \langle X,e_2\rangle^2)\\
&=\|AX\|^2 + \lambda_1\lambda_2\|X\|^2.
\end{split}
\]
 
\end{proof}

%The next Lemma is a version of the Bochner's formula, and we include here for the sake of completeness.

\begin{lemma}[Bochner's formula]\label{bochner}
Let $\Sigma$ be a Riemannian manifold of arbitrary dimension. If $\ric$ denotes the Ricci tensor of $\Sigma,$ then
\begin{equation}\label{bochner-eq}
\di(\n_VX)=\lan V,\n_\Sigma(\di X)\ran + \ric(V,X) + \tr(\n X \circ \n V)
\end{equation}
for every vector fields $V,X\in T\Sigma.$ Here $\n X:T\Sigma\ria T\Sigma$ is given by $(\n X)(u)=\n_u X$ and $\tr S$ denotes the trace of the linear operator $S:T\Sigma\ria T\Sigma.$
\end{lemma}
\begin{proof}
Denote by $m$ the dimension of $\Sigma.$ Fixing an arbitrary $p\in\Sigma,$ let $\{e_1,e_2,\ldots,e_m\}$ be an orthonormal frame of $T\Sigma,$ which is geodesic at $p.$ We have, at $p,$
\[
\begin{split}
\ric(V,X)&=\sum_{i=1}^m \lan R(V,e_i)X,e_i\ran\\
&=\sum_{i=1}^m\lan \n_{e_i}\n_VX,e_i\ran - \sum_{i=1}^m\lan \n_V\n_{e_1}X,e_i\ran +\sum_{i=1}^m\lan \n_{[V,e_i]}X,e_i\ran\\
&=\di(\n_VX) - \sum_{i=1}^n V\lan\n_{e_i}X,e_i\ran + \sum_{i=1}^m\lan \n_{\n_Ve_i - \n_{e_i}V}X,e_i\ran\\
&=\di(\n_VX) - V(\di X) - \sum_{i=1}^m\lan \n_{\n_{e_i}V}X,e_i\ran\\
&=\di(\n_VX) - \lan V,\n(\di X)\ran - \tr(\n X\circ \n V),\\
\end{split}
\]
since $\n_V e_i = \sum_{j=1}^m \lan V,e_j\ran \n_{e_j}e_i=0$ at $p.$ Here, $R$ denotes the curvature tensor of $\Sigma.$ The result then follows.
 
\end{proof}

Now we state the main proposition of this section.
\begin{proposition}\label{main-1}
Let $M^3$ be a three-dimensional Riemannian manifold isometrically immersed into $\R^n,\ n\geq4,$ and let $\Sigma$ be a constant mean curvature $H$ surface of $M^3.$ If $e_1,e_2$ is an orthonormal frame of $\Sigma$ which is geodesic at $p\in\Sigma,$ then at $p,$ we have
\begin{equation}\label{eq.prop.1}
\begin{split}
Q(X,X) + Q(JX,JX) &= (4H^2 + 6\scal_M)\|X\|^2\\
&\qquad - \sum_{\alpha=4}^n\sum_{i=1}^2\left[\overline{II}^\alpha(e_i,X)^2 + \overline{II}^\alpha(e_i,JX)^2\right],
\end{split}
\end{equation}
where $\scal_M$ is the normalized scalar curvature of $M$ and $\overline{II}^\alpha$ are the second fundamental form of $M^3$ in $\R^n$ associated with each normal $\overline{N}^\alpha, \ \alpha=4,\ldots,n.$
\end{proposition}

\begin{proof}
All the computations in this proof are made at $p\in\Sigma.$ Initially, let us calculate the Laplacian of $u_j.$ Since 
\[
\Delta_\Sigma u_j = \sum_{i=1}^2 \langle D_{e_i}D_{e_i} X, E_j\rangle,
\]
we need to calculate $D_{e_i}D_{e_i} X.$ Taking the first covariant derivative, we have
\[
\begin{split}
D_{e_i}X &= \overline{\n}_{e_i}X + \sum_{\alpha=4}^n\overline{II}^\alpha(e_i,X)\overline{N}^\alpha\\
         &= \n_{e_i}X + II(e_i,X)N + \sum_{\alpha=4}^n\overline{II}^\alpha(e_i,X)\overline{N}^\alpha, 
\end{split}
\]
which implies, by taking the covariant derivative again,
\[
\begin{split}
D_{e_i}D_{e_i} X&= D_{e_i}\n_{e_i} X + D_{e_i}(II(e_i,X)N) + \sum_{\alpha=4}^nD_{e_i}(\overline{II}^\alpha(e_i,X)\overline{N}^\alpha)\\
                &= \n_{e_i}\n_{e_i} X + II(e_i,\n_{e_i} X)N + \sum_{\alpha=4}^n\overline{II}^\alpha(e_i,\n_{e_i}X)\overline{N}^\alpha\\
                &\qquad + e_i(II(e_i,X))N + II(e_i,X)\overline{\n}_{e_i}N + II(e_i,X)\sum_{\alpha=4}^n\overline{II}^\alpha(e_i,N)\overline{N}^\alpha \\
                &\qquad + \sum_{\alpha=4}^n e_i(\overline{II}^\alpha(e_i,X))\overline{N}^\alpha+ \sum_{\alpha=4}^n\overline{II}^\alpha(e_i,X)D_{e_i}\overline{N}^\alpha\\
                &=\n_{e_i}\n_{e_i} X + [\langle Ae_i,\n_{e_i}X\rangle + e_i\langle e_i, AX\rangle] N - \langle AX,e_i\rangle Ae_i \\
                &\qquad + \sum_{\alpha=4}^n\left[\overline{II}^\alpha(e_i,\n_{e_i}X) + II(e_i,X)\overline{II}^\alpha(e_i,N)+e_i(\overline{II}^\alpha(e_i,X))\right]\overline{N}^\alpha\\
                &\qquad+ \sum_{\alpha=4}^n\overline{II}^\alpha(e_i,X)D_{e_i}\overline{N}^\alpha,\\
\end{split}
\]
because $D_{e_i}N = \overline{\n}_{e_i}N + \sum_{\alpha=4}^n\overline{II}^\alpha(e_i,N)\overline{N}^\alpha.$ This implies
\[
\begin{split}
\sum_{j=1}^n u_j\Delta_\Sigma u_j &= \sum_{i=1}^2\sum_{j=1}^n \langle D_{e_i}D_{e_i} X, E_j\rangle\langle E_j,X\rangle\\
                                  &= \sum_{i=1}^2\langle \n_{e_i}\n_{e_i} X,X\rangle - \sum_{i=1}^2 \langle AX,e_i\rangle\langle Ae_i,X\rangle\\
                                  &\qquad +\sum_{\alpha=4}^n\sum_{i=1}^2 \overline{II}^\alpha(e_i,X)\langle D_{e_i}\overline{N}^\alpha,X\rangle,\\   
\end{split}
\]
provided 
\[
\sum_{j=1}^n \langle N,E_j\rangle\langle E_j,X\rangle= \langle N,X\rangle=0 \ \mbox{and} \ \sum_{j=1}^n \langle \overline{N}^\alpha,E_j\rangle\langle E_j,X\rangle= \langle \overline{N}^\alpha,X\rangle=0.
\] 
Since
\[
\sum_{i=1}^2 \langle AX,e_i\rangle\langle Ae_i,X\rangle = \sum_{i=1}^2 \langle AX,e_i\rangle\langle e_i,AX\rangle = \|AX\|^2
\]
and
\[
\langle D_{e_i}\overline{N}^\alpha,X\rangle= -\overline{II}^\alpha(e_i,X),
\]
we obtain
\begin{equation}\label{u-lap-u}
\sum_{j=1}^n u_j\Delta_\Sigma u_j = \sum_{i=1}^2\langle \n_{e_i}\n_{e_i} X,X\rangle - \|AX\|^2 - \sum_{\alpha=4}^n\sum_{i=1}^2 \overline{II}^\alpha(e_i,X)^2.
\end{equation}
Now, let us calculate a more suitable expression for $\sum_{i=1}^2\langle \n_{e_i}\n_{e_i} X,X\rangle.$  Since $X$ is harmonic, we have $\langle\n_V X,e_i\rangle =\langle\n_{e_i}X,V\rangle$ for every $V\in T\Sigma.$ By using that $e_1,e_2$ is a geodesic frame at $p$, this implies
\[
\begin{split}
\di(\n_VX)&=\sum_{i=1}^2\langle\n_{e_i}\n_VX,e_i\rangle =\sum_{i=1}^2e_i\langle\n_VX,e_i\rangle =\sum_{i=1}^2e_i\langle\n_{e_i}X,V\rangle\\
          &=\sum_{i=1}^2\langle\n_{e_i}\n_{e_i}X,V\rangle + \sum_{i=1}^2\langle\n_{e_i}X,\n_{e_i}V\rangle,\\
\end{split}
\]
i.e.,
\begin{equation}\label{eq.mmm}
\sum_{i=1}^2\langle\n_{e_i}\n_{e_i}X,V\rangle = \di(\n_VX) - \sum_{i=1}^2\langle\n_{e_i}X,\n_{e_i}V\rangle.
\end{equation}
On the other hand, since $\Sigma$ has dimension $2,$ the Bochner's formula (\ref{bochner-eq}), p. \pageref{bochner-eq}, becomes
\[
\di(\n_VX) =\langle V,\n(\di X)\rangle + K\langle X,V\rangle + \tr(\n X\circ\n V),\\
\]
where $K$ is the Gaussian curvature of $\Sigma.$ Since $X$ is harmonic, using Lemma \ref{harm-field}, p. \pageref{harm-field}, we have
\[
\di X =0, \ \mbox{and} \ \tr(\n X\circ\n V)= \sum_{i=1}^2 \lan \n_{\n_{e_i}V}X,e_i \ran = \sum_{i=1}^2\lan \n_{e_i}X,\n_{e_i}V\ran.
\]
Replacing the last two equations in the Bochner's formula we obtain
\[
\di(\n_VX) =K\langle X,V\rangle + \sum_{i=1}^2\langle\n_{e_i}X,\n_{e_i}V\rangle,
\]
which implies, by using equation (\ref{eq.mmm}), that
\[
\sum_{i=1}^2\langle\n_{e_i}\n_{e_i}X,V\rangle =K\langle X,V\rangle.
\]
Thus, equation (\ref{u-lap-u}) becomes
\begin{equation}\label{eq.lap}
\sum_{j=1}^n u_j\Delta_\Sigma u_j = K\|X\|^2 - \|AX\|^2 - \sum_{\alpha=4}^n\sum_{i=1}^2 \overline{II}^\alpha(e_i,X)^2.
\end{equation}
By using Lemma \ref{previous}, p. \pageref{previous}, in (\ref{eq.lap}), we have
\[
\sum_{j=1}^n u_j\Delta_\Sigma u_j = (K+K_e)\|X\|^2 - 2H\langle AX,X\rangle - \sum_{\alpha=4}^n\sum_{i=1}^2 \overline{II}^\alpha(e_i,X)^2.
\]
Since, by the Gauss Equation, $K_e = K - \overline{K}(e_1,e_2),$ where $\overline{K}(e_1,e_2)$ is the sectional curvature of $M^3$ in $T\Sigma,$ we obtain
\[
\sum_{j=1}^n u_j\Delta_\Sigma u_j = (2K_e + \overline{K}(e_1,e_2))\|X\|^2 - 2H\langle AX,X\rangle - \sum_{\alpha=4}^n\sum_{i=1}^2 \overline{II}^\alpha(e_i,X)^2.
\]
By using $4H^2 = \|A\|^2 + 2K_e,$ we have
\begin{equation}\label{eq.lap-final}
\sum_{j=1}^n u_j\Delta_\Sigma u_j = (4H^2 - \|A\|^2 + \overline{K}(e_1,e_2))\|X\|^2 - 2H\langle AX,X\rangle - \sum_{\alpha=4}^n\sum_{i=1}^2 \overline{II}^\alpha(e_i,X)^2.
\end{equation}
Therefore
\[
\begin{split}
Q(X,X)&=(4H^2 - \|A\|^2 + \overline{K}(e_1,e_2))\|X\|^2 - 2H\langle AX,X\rangle - \sum_{\alpha=4}^n\sum_{i=1}^2 \overline{II}^\alpha(e_i,X)^2\\
&\qquad + (\ric_M (N,N) + \|A\|^2)\|X\|^2\\
&= (4H^2 + 3\scal_M)\|X\|^2  - 2H\langle AX,X\rangle - \sum_{\alpha=4}^n\sum_{i=1}^2 \overline{II}^\alpha(e_i,X)^2,\\
\end{split}
\]
provided 
\[\overline{K}(e_1,e_2) + \ric_M (N,N) = \overline{K}(e_1,e_2) + \overline{K}(e_1,N) + \overline{K}(e_2,N) = 3\scal_M\]
and
\[
\sum_{j=1}^n u_j^2 = \sum_{j=1}^n \langle X,E_j\rangle^2 = \|X\|^2,
\]
where $\overline{K}(e_i,N),$ $i=1,2,$ is sectional curvature of $M^3$ in the plane spanned by $e_i$ and $N.$ Analogously,
\[
Q(JX,JX) =\left(4H^2 + 3\scal_M\right)\|JX\|^2  - 2H\langle AJX,JX\rangle - \sum_{\alpha=4}^n\sum_{i=1}^2 \overline{II}^\alpha(e_i,JX)^2.
\]
Note that, since $\|JX\|=\|X\|$ and $\langle JX,X\rangle=0,$ i.e., $X$ and $JX$ is an orthogonal frame of $\Sigma$, we have
\[
\begin{split}
2H\langle AX,X\rangle + 2H\langle AJX,JX\rangle &= 2H\left[\left\langle A\dfrac{X}{\|X\|},\dfrac{X}{\|X\|}\right\rangle + \left\langle A\dfrac{JX}{\|JX\|}\!,\!\dfrac{JX}{\|JX\|}\right\rangle \right]\|X\|^2\\
& = 2H (\trace A) \|X\|^2\\
& = 4H^2\|X\|^2.\\
\end{split}
\]
This implies
\[
\begin{split}
Q(X,X) + Q(JX,JX) =& 2\left(4H^2 + 3\scal_M\right)\|X\|^2 - 4H^2\|X\|^2\\
                   &- \sum_{\alpha=4}^n\sum_{i=1}^2[ \overline{II}^\alpha(e_i,X)^2 + \overline{II}^\alpha(e_i,JX)^2]\\
                   &= (4H^2 + 6\scal_M)\|X\|^2 \\
                   &\qquad - \sum_{\alpha=4}^n\sum_{i=1}^2[ \overline{II}^\alpha(e_i,X)^2 + \overline{II}^\alpha(e_i,JX)^2].\\
\end{split}
\]
 
\end{proof}

\section{Proof of Theorems \ref{theo-warped-1} and \ref{theo-warped-2}}

In order to prove Theorems \ref{theo-warped-1} and \ref{theo-warped-2}, let us look the warped product manifold as a hypersurface of the Euclidean space of the Lorentzian space.

Let $\mathbb{L}^{n+2}$ be $\R^{n+2}$ with the (pseudo)metric $\langle\cdot,\cdot\rangle = \kappa dx_0^2 + dx_1^2+\cdots+dx_{n+1}^2,$ where $\kappa=\pm1.$ If $\kappa=1$ then $\mathbb{L}^{n+2}$ is just $\R^{n+2}$ with the canonical metric, and if $\kappa=-1$ then $\mathbb{L}^{n+2}$ is the Lorentzian space with its usual pseudo-metric.

Fixed the smooth function $h:I\rightarrow\R,$ where $I=(0,b)$ or $I=(0,\infty),$ let us define $f:I\rightarrow\R$ by the equation
\[
\kappa f'(t)^2 + h'(t)^2=1.
\] 
Let $\om:\R^n\rightarrow\R^{n+1}$ be the canonical immersion of the unit sphere $\s^n$ in polar coordinates, i.e., for $\theta=(\vp_1,\vp_2,\ldots,\vp_n)\in\R^n,$ we have
\[
\om(\theta)=(x_1(\theta),x_2(\theta),\ldots,x_{n+1}(\theta)),
\] 
where
\[
\begin{split}
x_1&=\cos(\vp_1),\\
x_2&=\sin(\vp_1)\cos(\vp_2),\\
x_3&=\sin(\vp_1)\sin(\vp_2)\cos(\vp_3),\\
\vdots&\\
x_n&=\sin(\vp_1)\cdots \sin(\vp_{n-1})\cos(\vp_n),\\
x_{n+1}&=\sin(\vp_1)\cdots \sin(\vp_{n-1})\sin(\vp_n).\\
\end{split}
\]
Consider $F:I\times\R^n\subset\R^{n+1}\rightarrow\mathbb{L}^{n+2}$ be the immersion 
\[
F(t,\theta)=(f(t),h(t)\om(\theta)).
\] 
Denoting by $F_t = \dfrac{\partial F}{\partial t},$ $F_i = \dfrac{\partial F}{\partial \vp_i},$ and $\om_i=\dfrac{\partial \om}{\partial\vp_i},$ we have
\[
F_t=(f'(t),h'(t)\omega) \ \mbox{and} \ F_i = (0,h(t)\om_i).
\]
Since $\langle \om,\om_i\rangle=0$ and $\langle \om_i,\om_j\rangle=0$ for $i\neq j$, the first fundamental form of this immersion is
\[
\langle\cdot,\cdot\rangle = dt^2 + h(t)^2d\om^2,
\]
where $d\om^2 = \sum_{i=1}^n \|\om_i\|^2 d\vp_i^2$ is the first fundamental form of $\s^n$ parametrized by $\om.$ Next let us find the second fundamental form of the immersion $F.$ 

\begin{proposition}\label{II-warped}
The second fundamental form of $M^{n+1}=F(I\times\R^n)$ in $\mathbb{L}^{n+2}$ is given by
\[
\overline{II} = \kappa \sqrt{|K_{\tan}(t)|}\langle\cdot,\cdot\rangle - \dfrac{K_{\tan}(t) - K_{\rad}(t)}{\sqrt{|K_{\tan}(t)|}}dt^2.
\]
\end{proposition}
\begin{proof}
Consider
\[
E_0=F_t, \ E_i=\left(0,\dfrac{\om_i}{\|\om_i\|}\right), \ E_{n+1}=(-h'(t),\kappa f'(t)\om)
\]
be an orthonormal frame of $\mathbb{L}^{n+2}$ and $\alpha_i,\ i=0,1,\ldots,n+1,$ its dual frame, i.e., linear functionals such that $\alpha_i(E_i)=1$ and $\alpha_i(E_j)=0$ if $i\neq j.$ Since
\[
dF = (f'(t)dt,h'(t)\om dt + h(t) d\om),
\]
we have
\begin{equation}\label{alpha-i}
\begin{split}
\alpha_0 &= \langle dF,E_0 \rangle = \kappa f'(t)^2dt + h'(t)^2 dt = dt,\\
\alpha_i &= \langle dF,E_i\rangle = \dfrac{h(t)}{\|\om_i\|}\langle d\om,\om_i\rangle,\\
\alpha_{n+1} &= \langle dF,E_{n+1}\rangle=0.\\
\end{split}
\end{equation}
Now let us calculate the second structure forms. We have
\[
\begin{split}
dE_0 &= (f''(t)dt,h''(t)\om dt + h'(t)d\om),\\
dE_i &= \left(0,d\left(\dfrac{\om_i}{\|\om_i\|}\right)\right),\\
dE_{n+1}&= (-h''(t)dt, \kappa f''(t)\om dt +\kappa f'(t)d\om).\\
\end{split}
\]
This implies
\[
\begin{split}
\alpha_{0;i} &= \langle dE_0,E_i\rangle = \dfrac{h'(t)}{\|\om_i\|}\langle d\om,\om_i\rangle = \dfrac{h'(t)}{h(t)}\alpha_i,\\
\alpha_{0;n+1}&= \langle dE_0,E_{n+1}\rangle = \kappa(h''(t)f'(t) - f''(t)h'(t))dt \\
			  &= \kappa(h''(t)f'(t) - f''(t)h'(t))\alpha_0,\\
\alpha_{n+1;i} &= \langle dE_{n+1},E_i \rangle = \dfrac{\kappa f'(t)}{\|\om_i\|}\langle d\om,\om_i\rangle = \dfrac{\kappa f'(t)}{h(t)}\alpha_i.  
\end{split}
\]
On the other hand,
\[
\alpha_{n+1;0} = \sum_{p=0}^n h_{0p}\alpha_p \ \mbox{and} \ \alpha_{n+1;i} = \sum_{p=0}^n h_{ip}\alpha_p.
\]
This implies
\[
\begin{split}
h_{00}=\kappa(f''(t)h'(t) - f'(t)h''(t)),& \ h_{0p}=0, \ p=1,\ldots,n,\\
\ h_{ii}=\kappa\dfrac{f'(t)}{h(t)}, \ i=1,\ldots,n, &\ h_{ip}=0, \ p\neq i.\\
\end{split}
\]
Thus
\begin{equation}\label{II-part-1}
\overline{II} = \sum_{i,j=0}^n h_{ij}\alpha_i\alpha_j = \kappa(f''(t)h'(t) - f'(t)h''(t))dt^2 + \kappa\dfrac{f'(t)}{h(t)}\sum_{j=1}^n\alpha_j^2.
\end{equation}
Since $d\om=\sum_{j=1}^n \om_jd\vp_j$ and by using (\ref{alpha-i}), we have 
\[
\alpha_i = \frac{h(t)}{\|\om_i\|}\sum_{i=1}^n\lan\om_j,\om_i\ran d\vp_j = h(t)\|\om_i\|d\vp_i.
\]
This implies
\begin{equation}\label{metric-alpha}
\begin{split}
\lan\cdot,\cdot\ran&=dt^2 + h(t)^2d\om^2 = dt^2 + h(t)^2\sum_{i=1}^n\|\om_i\|^2d\vp_i^2\\
                   &= dt^2+\sum_{i=1}^n\alpha_i^2.\\
\end{split}
\end{equation}
Replacing (\ref{metric-alpha}) in (\ref{II-part-1}), we obtain
\begin{equation}\label{II-part-2}
\begin{split}
\overline{II} &=\kappa(f''(t)h'(t) - f'(t)h''(t))dt^2 + \kappa\dfrac{f'(t)}{h(t)}(\langle\cdot,\cdot\rangle - dt^2)\\
&= \kappa\dfrac{f'(t)}{h(t)}\langle\cdot,\cdot\rangle + \kappa\left(f''(t)h'(t) - f'(t)h''(t) - \dfrac{f'(t)}{h(t)}\right)dt^2.\\
\end{split}
\end{equation}
Note that 
\[
\kappa f'(t)^2 = 1-h'(t)^2 = K_{\tan}(t)h(t)^2
\]
and, by taking derivatives, 
\[
\kappa f'(t)f''(t) = -h'(t)h''(t).
\]
This implies
\[
\begin{split}
f'(t)[f''(t)h'(t) - f'(t)h''(t)]&=f'(t)f''(t)h'(t) - f'(t)^2h''(t)\\
                         &=-\kappa h''(t)h'(t)^2 - \kappa(1-h'(t)^2)h''(t)\\
                         &=-\kappa h''(t)\\
                         &=\kappa K_{\rad}(t)h(t).\\
\end{split}
\]
Thus, replacing the last equation in (\ref{II-part-2}), we obtain
\[
\dfrac{f'(t)}{h(t)}\overline{II} = \kappa \left(\dfrac{f'(t)}{h(t)}\right)^2\langle\cdot,\cdot\rangle + \kappa\left(\kappa K_{\rad}(t) - \left(\dfrac{f'(t)}{h(t)}\right)^2\right)dt^2. 
\]
Since $f'(t)^2 = \kappa K_{\tan}(t)h(t)^2 = |K_{\tan}(t)|h(t)^2,$ we have 
\[
\sqrt{|K_{\tan}(t)|}\overline{II} = K_{\tan}(t)\langle\cdot,\cdot\rangle - (K_{\tan}(t) - K_{\rad}(t))dt^2,
\]
i.e.,
\[
\overline{II} = \kappa \sqrt{|K_{\tan}(t)|}\langle\cdot,\cdot\rangle - \dfrac{K_{\tan}(t) - K_{\rad}(t)}{\sqrt{|K_{\tan}(t)|}}dt^2.
\]
 
\end{proof}

The main part of the proof of Theorems \ref{theo-warped-1} and \ref{theo-warped-2} is the following two propositions, which have their own interest.

\begin{proposition}\label{prop-epsilon}
Let $\Sigma$ be a compact, without boundary, stable, constant mean curvature $H\neq0$ surface of a three-dimensional Riemannian manifold $M^3$ and let $\overline{II}=a\langle\cdot,\cdot\rangle + \ve a\langle\cdot,\xi\rangle\langle\cdot,\xi\rangle$ be the second fundamental form of $M^3$ in $\R^4,$ where $\xi$ is a unitary vector field of $TM$ and $a,\ve:M\rightarrow\R$ are smooth functions. If 
\[
-1\leq\ve\leq 1+\sqrt{5},
\]
then $\genus(\Sigma)=0.$
\end{proposition}
\begin{proof}

If $\Sigma$ is a compact, without boundary, stable, constant mean curvature $H\neq0$ surface such that $\genus(\Sigma)\geq 1$ then there exist at least two harmonic vector fields $X$ and $JX$ such that
\[
0\leq - \int_\Sigma \left(Q(X,X)+Q(JX,JX)\right) d\Sigma.
\]
If we prove that $Q(X,X)+Q(JX,JX)\geq 0$ under the assumptions, it will be possible to prove that $H=0,$ which is a contradiction. This will give us there is no compact, without boundary, stable, constant mean curvature $H\neq0$ surfaces with $\genus(\Sigma)\geq 1$ in $M^3,$ which implies that $\genus(\Sigma)=0.$

Let $\{e_1,e_2,e_3\}$ be an adapted frame of $M^3$ (i.e., such that $e_1,e_2\in T\Sigma$ and $e_3=N$). Assume also that $\{e_1,e_2\}$ is a geodesic frame at $p\in\Sigma.$ By the Gauss equation, we have, at $p,$
\[
\begin{split}
6\scal_M &= \sum_{i,j=1}^3[\overline{II}(e_i,e_i)\overline{II}(e_j,e_j) - \overline{II}(e_i,e_j)^2]\\
        &= \sum_{i,j=1}^3\left[(a\langle e_i,e_i\rangle + \ve a\langle e_i,\xi\rangle\langle e_i,\xi\rangle)(a\langle e_j,e_j\rangle + \ve a\langle e_j,\xi\rangle\langle e_j,\xi\rangle)\right.\\
        &\left.\qquad - (a\langle e_i,e_j\rangle + \ve a\langle e_i,\xi\rangle\langle e_j,\xi\rangle)^2\right]\\
        &=a^2 \sum_{i,j=1}^3 [\langle e_i,e_i\rangle\langle e_j,e_j\rangle -\langle e_i,e_j\rangle^2]\\
        &\qquad + \ve a^2\sum_{i,j=1}^3[\langle e_i,e_i\rangle \langle e_j,\xi\rangle^2 + \langle e_j,e_j\rangle \langle e_i,\xi\rangle^2 - 2\langle e_i,e_j\rangle \langle e_i,\xi\rangle \langle e_j,\xi\rangle ]\\
        &=6a^2 + \ve a^2[3\sum_{j=1}^3\langle e_j,\xi\rangle^2 + 3\sum_{i=1}^3\langle e_i,\xi\rangle^2 - 2\sum_{i=1}^3\langle e_i,\xi\rangle^2]\\
        &=6a^2 + 4\ve a^2|\xi|^2\\
        &=6a^2+4\ve a^2.
\end{split} 
\]
On the other hand, by using that $\|JX\|=\|X\|$ and $\langle X,JX\rangle=0,$ i.e., $X$ and $JX$ is an orthogonal frame for $\Sigma,$ we have, at $p,$

\[
\begin{split}
\sum_{i=1}^2 \overline{II}(e_i,X)^2 &+ \overline{II}(e_i,JX)^2 = \sum_{i=1}^2 (a\langle e_i,X\rangle + \ve a \langle  e_i,\xi\rangle\langle X,\xi\rangle)^2 \\
&\qquad+ (a\langle e_i,JX\rangle + \ve a \langle  e_i,\xi\rangle\langle JX,\xi\rangle)^2\\
&=a^2\sum_{i=1}^2\left(\langle X,e_i\rangle^2 + 2\ve\langle X,e_i\rangle\langle e_i,\xi\rangle\langle X,\xi\rangle + \ve^2\langle X,\xi\rangle^2\langle \xi,e_i\rangle^2\right)\\
&+a^2\sum_{i=1}^2\left(\langle JX,e_i\rangle^2 + 2\ve\langle JX,e_i\rangle\langle e_i,\xi\rangle\langle JX,
\xi\rangle + \ve^2\langle JX,\xi\rangle^2\langle \xi,e_i\rangle^2\right)\\
&=a^2(\|X\|^2 + 2\ve\langle X,\xi\rangle^2 + \ve^2\langle X,\xi\rangle^2(1-\nu^2))\\ 
&\qquad +a^2(\|JX\|^2 + 2\ve\langle JX,\xi\rangle^2 + \ve^2\langle JX,\xi\rangle^2(1-\nu^2))\\
&=a^2 (2 + 2\ve (1-\nu^2) + \ve^2(1-\nu^2)^2)\|X\|^2,
\end{split}
\]
provided
\[
\sum_{i=1}^2 \langle\xi,e_i\rangle^2 = 1-\langle N,\xi\rangle^2 = 1-\nu^2,
\]
and
\[
\langle X,\xi\rangle^2 + \langle JX,\xi\rangle^2 = (1-\langle \xi, N\rangle^2)\|X\|^2=(1-\nu^2)\|X\|^2,
\]
where $\nu = \langle N,\xi\rangle.$ Thus, using Proposition \ref{main-1} we get
\begin{equation}\label{Q-1}
\begin{split}
Q(X,X)+Q(JX,JX) &= (4H^2 + 6a^2 + 4\ve a^2)\|X\|^2\\
                &\qquad - a^2\left(2 + 2\ve (1-\nu^2) + \ve^2(1-\nu^2)^2\right)\|X\|^2\\
                &= \left(4H^2 + a^2\left(4 + 4\ve - 2\ve(1-\nu^2) - \ve^2(1-\nu^2)^2\right)\right)\|X\|^2.
\end{split}
\end{equation}
Since equation (\ref{Q-1}) does not depend on $e_1$ and $e_2,$ we have that it does not depend on $p,$ and thus it holds everywhere in $\Sigma.$ In order to prove that $Q(X,X)+Q(JX,JX)\geq 0,$ we will find the values of $\ve$ such that the expression (\ref{Q-1}) is non-negative for every $H>0$. This means that we need to find some conditions for $\ve$ such that 
\[
4(1 + \ve) - 2\ve(1-\nu^2) - \ve^2(1-\nu^2)^2\geq 0
\]
for all values of $\nu\in[-1,1].$ Let 
\[
p(y)=4(1 +\ve) - 2\ve y - \ve^2 y^2.
\] 
Let us prove that $p(y)\geq 0$ for every $y\in[0,1]$ if $-1\leq \ve \leq 1+\sqrt{5}.$ First note that it holds trivially for $\ve=0.$ Let us analyze the case $\ve>0.$ In this case $p(y)$ has the roots
\[
y_1 = \dfrac{-1-\sqrt{5+4\ve}}{\ve} \ \mbox{and} \ y_2 = \dfrac{\sqrt{5+4\ve}-1}{\ve}. 
\]
Since we want $p(y)\geq0$ for every $y\in[0,1],$ we need $y_1\leq 0$ and $y_2\geq 1.$ Observe that $y_1<0$ for every $\ve>0.$ On the other hand,
\[
y_2\geq 1 \Leftrightarrow \sqrt{5+4\ve}\geq \ve +1 \Leftrightarrow 5+4\ve \geq (\ve+1)^2 \Leftrightarrow \ve^2 - 2\ve - 4\leq 0 \Leftrightarrow \ve \leq 1+ \sqrt{5}.
\]
Thus, if $\ve>0,$ then $p(y)\geq 0$ for every $y\in[0,1]$ if $\ve\leq 1+\sqrt{5}.$ 

On the other hand, if $\ve<0$ we can consider $\ve = -|\ve|.$ In this case, the roots of $p(y)$ are
\[
y_1 = \dfrac{1-\sqrt{5-4|\ve|}}{|\ve|} \ \mbox{and}\ y_2 = \dfrac{1+\sqrt{5-4|\ve|}}{|\ve|}.
\]
Note that $|\ve|\leq 5/4,$ i.e., $\ve\geq -5/4$ is the first restriction for $\ve<0.$ We have
\[
y_1\leq0 \Leftrightarrow 1-\sqrt{5-4|\ve|}\leq 0 \Leftrightarrow \sqrt{5-4|\ve|}\geq 1 \Leftrightarrow 5-4|\ve|\geq 1 \Leftrightarrow |\ve|\leq 1,
\]
i.e., $y_1\leq 0$ for $\ve\geq-1.$ On the other hand,
\[
y_2\geq 1 \Leftrightarrow \sqrt{5-4|\ve|}+1 \geq |\ve| \Leftrightarrow \sqrt{5-4|\ve|}\geq |\ve|-1 
\]
and this is true since $|\ve|\leq 1$ implies $\sqrt{5-4|\ve|}\geq 0 \geq |\ve|-1.$ Thus, if $\ve<0,$ then $p(y)\geq 0$ for all $y\in[0,1]$ if $|\ve|\leq 1,$ i.e., $\ve\geq -1.$ The combination of both cases gives that $p(y)\geq 0$ for $\ve\in[-1,1+\sqrt{5}].$ Thus,
\[
0\leq -\int_\Sigma \left(Q(X,X) + Q(JX,JX)\right) d\Sigma = -\int_\Sigma (4H^2 + p(y)) d\Sigma \leq 0,
\]
which implies that $4H^2 + p(y)\equiv 0$ i.e., $H=0,$ which is a contradiction. 
 
\end{proof}

If $\ve\not\in[-1,1+\sqrt{5}],$ we can determine the values of $H>0$ such that compact, without boundary, constant mean curvature $H\neq0$ surfaces have genus zero for given $a$ and $\ve.$

\begin{proposition}\label{prop-epsilon-2}
Let $\Sigma$ be a compact, without boundary, stable, constant mean curvature $H\neq0$ surface of a three-dimensional Riemannian manifold $M^3$ and let $\overline{II}=a\langle\cdot,\cdot\rangle + \ve a\langle\cdot,\xi\rangle\langle\cdot,\xi\rangle$ be the second fundamental form of $M^3$ in $\R^4,$ where $\xi$ is a unitary vector field of $TM$ and $a,\ve:M\rightarrow\R,$ $a\neq0,$ are smooth functions. If one of the following conditions is satisfied
\begin{itemize}
\item[i)] $\ve>1+\sqrt{5}$ and $H^2 > \sup_\Sigma\dfrac{a^2(\ve^2 - 2\ve - 4)}{4};$

\item[ii)] $-2\leq \ve<-1$ and $H^2> \sup_\Sigma a^2(|\ve|-1);$

\item[iii)] $\ve<-2$ and $H^2> \sup_\Sigma a^2\left\{\frac{1}{4}|\ve|^2 + \frac{1}{2}|\ve|-1\right\},$ 
\end{itemize}
then $\genus(\Sigma)=0.$
\end{proposition}
\begin{proof}
In order to simplify the analysis, denote by $H_a=H/a.$ This implies by (\ref{Q-1}) that
\begin{equation}\label{Q-2}
Q(X,X) + Q(JX,JX) = a^2\left[4(H_a^2 + 1 + \ve) - 2\ve(1-\nu^2) - \ve^2(1-\nu^2)^2\right]\|X\|^2.
\end{equation}
Following the same idea of the previous proposition, we will find conditions such that (\ref{Q-2}) is non-negative. Let $p_a:[0,1]\ria\R$ be defined by
\[
p_a(y) = 4(H_a^2 +1 +\ve) - 2\ve y - \ve^2y^2.
\] 
If $\ve>0,$ then $p_a(y)$ has the roots
\[
y_1 = \dfrac{-1-\sqrt{5+4\ve + 4H_a^2}}{\ve} \ \mbox{and} \ y_2=\dfrac{-1+\sqrt{5+4\ve +4H_a^2}}{\ve}.
\]
Notice that $y_1<0$ for every $\ve>0$ and for every $H_a.$ On the other hand 
\[
y_2 > 1 \Leftrightarrow H_a^2 \geq \dfrac{\ve^2 - 2\ve - 4}{4}.
\]
Thus, if $\ve>0,$ then (\ref{Q-2}) is positive for
\[
H^2 > \dfrac{a^2(\ve^2 - 2\ve - 4)}{4}.
\]
If $\ve< 0,$ the roots of $p_a(y)$ are
\[
y_1= \dfrac{1- \sqrt{5-4|\ve|+4H_a^2}}{|\ve|} \ \mbox{and} \ y_2 = \dfrac{1+\sqrt{5-4|\ve|+4H_a^2}}{|\ve|}.
\]
First note that we need $H_a^2\geq |\ve| - \frac{5}{4}.$ This implies
\[
y_1 < 0 \Leftrightarrow H_a^2 > |\ve|-1 \ \mbox{and} \ y_2 > 1 \Leftrightarrow H_a^2 > \dfrac{|\ve|^2 +2|\ve| -4}{4}.
\]
Thus, if $\ve<0$ then (\ref{Q-2}) is positive for 
\[
H^2 > a^2 \max\left\{|\ve|-1,\dfrac{|\ve|^2 + 2|\ve|-4}{4}\right\}=
\left\{\begin{array}{lr}
a^2(|\ve|-1)&\mbox{if} \ |\ve|\leq 2;\\
a^2\left(\dfrac{|\ve|^2 + 2|\ve|-4}{4}\right) &\mbox{if} \ |\ve|\geq 2.\\
\end{array}\right.
\]
This implies, under the hypothesis, that $Q(X,X)+Q(JX,JX)> 0.$ Analogously to the proof of Proposition \ref{prop-epsilon}, if $\Sigma$ is a compact, without boundary, stable, nonzero constant mean curvature surface with $\genus(\Sigma)\geq 1,$ then there exists harmonic vector fields $X$ and $JX$ such that
\[
0\leq - \int_\Sigma \left(Q(X,X) + Q(JX,JX)\right)d\Sigma = -\int_\Sigma p_a(1-\nu^2)d\Sigma < 0,
\]
which gives a contradiction. Therefore $\genus(\Sigma)=0.$
 
\end{proof}

If we take $\ve=\delta a^{-2},$ where $\delta:M\rightarrow\R$ is a smooth function, we can rewrite Propositions \ref{prop-epsilon} and \ref{prop-epsilon-2} as 
\begin{corollary}\label{cor-delta} 
Let $\Sigma$ be a compact, without boundary, stable, constant mean curvature $H\neq0$ surface of a three-dimensional Riemannian manifold $M^3$ and let $\overline{II}=a\langle\cdot,\cdot\rangle + \delta a^{-1}\langle\cdot,\xi\rangle\langle\cdot,\xi\rangle$ be the second fundamental form of $M^3$ in $\R^4,$ where $\xi$ is a unitary vector field of $TM$ and $a,\delta:M\rightarrow\R,$ $a\neq0,$ are smooth functions. If one of the following conditions is satisfied 
\begin{itemize}
\item[i)] $-a^2\leq \delta \leq a^2(1+\sqrt{5});$

\item[ii)] $\delta\geq a^2(1+\sqrt{5})$ and $H^2>\sup_\Sigma\left\{\frac{1}{4}a^{-2}\delta^2 - \frac{1}{2}\delta  - a^2\right\};$

\item[iii)] $-2a^2\leq\delta\leq -a^2$ and $H^2 > \sup_\Sigma \{|\delta| - a^2\};$

\item[iv)] $\delta\leq -2a^2$ and $H^2>\sup_\Sigma\left\{\frac{1}{4}|\delta|^2a^{-2} + \frac{1}{2}|\delta|-a^2\right\},$
\end{itemize}
then $\genus(\Sigma)=0.$
\end{corollary}
 
\begin{proof}[Proof of Theorem \ref{theo-warped-1}.] Applying Corollary \ref{cor-delta} (i) to the warped product manifold with the second fundamental form of Proposition \ref{II-warped}, p. \pageref{II-warped}, by considering $\kappa=1,$ $a(t) =\sqrt{|K_{\tan}(t)|}$ and $\delta(t)=K_{\rad}(t) - K_{\tan}(t)$ we prove that $\genus(\Sigma)=0.$ 

In order to prove that $\Sigma$ is embedded, we will prove that the hypothesis of Theorem \ref{theo-warped-1} implies the hypothesis of Theorem \ref{Stab-Main}, p. \pageref{Stab-Main}. If $K_{\rad}(t)\geq K_{\tan}(t)>0,$ then $\Sigma$ is embedded for all $H>0$ by using Theorem \ref{Stab-Main} (ii). If $K_{\tan}(t)\geq K_{\rad}(t)\geq 0$ then we can apply Theorem \ref{Stab-Main} (i) to conclude that $\Sigma$ is embedded for every $H>0.$
 
\end{proof}

\begin{proof}[Proof of Theorem \ref{theo-warped-2}.]

Applying Corollary \ref{cor-delta} (ii), (iii) and (iv), to the warped product manifold with the second fundamental form of Proposition \ref{II-warped}, p. \pageref{II-warped}, by considering $\kappa=1,$ $a(t) =\sqrt{|K_{\tan}(t)|}$ and $\delta(t)=K_{\rad}(t) - K_{\tan}(t)$ we prove that $\genus(\Sigma)=0$ for the following values of $c_0=c_0(M)$ (see Fig. \ref{fig-3}):

\begin{itemize}
\item[(a)] If $0<(2+\sqrt{5})K_{\tan}(t)\leq K_{\rad}(t),$ then 
\[
c_0(M)=\dfrac{1}{4}\sup_{t\in I}\left\{K_{\tan}(t)\left[\left(\dfrac{K_{\rad}(t)}{K_{\tan}(t)}\right)^2 - 4\left(\dfrac{K_{\rad}(t)}{K_{\tan}(t)}\right) -1\right]\right\};
\]

\item[(b)] If $0<-K_{\rad}(t)\leq K_{\tan}(t),$ then
\[
c_0(M)=\sup_{t\in I}\{- K_{\rad}(t)\};
\]

\item[(c)] If $-K_{\rad}(t)\geq K_{\tan}(t)>0,$ then
\[
c_0(M)=\dfrac{1}{4}\sup_{t\in I}\left\{K_{\tan}(t)\left[\left(\dfrac{-K_{\rad}(t)}{K_{\tan}(t)}\right)^2 + 4\left(\dfrac{-K_{\rad}(t)}{K_{\tan}(t)}\right)-1\right]\right\}.
\]
\end{itemize}

In order to prove that $\Sigma$ is embedded, we will prove that the hypothesis of Theorem \ref{theo-warped-2}, with the values of $c_0(M)$ stated in the items (a) to (c) above, implies the hypothesis of Theorem \ref{Stab-Main}, p. \pageref{Stab-Main}. We prove each case separately:

\begin{itemize}
\item[(a)] If $K_{\rad}(t)\geq (2+\sqrt{5})K_{\tan}(t)>K_{\tan}(t)>0$ then, by using Theorem \ref{Stab-Main} (ii) we can see that $\Sigma$ is embedded for every $H>0.$

\item[(b)] The value of $c_0(M)$ is identical to the hypothesis of Theorem \ref{Stab-Main} (i).

\item[(c)] Since $K_{\tan}(t)>0> K_{\rad}(t),$  and 
\[
\begin{split}
\dfrac{1}{4}K_{\tan}(t)\left(\left(\dfrac{-K_{\rad}(t)}{K_{\tan}(t)}\right)^2 +4\left(\dfrac{-K_{\rad}(t)}{K_{\tan}(t)}\right)-1\right) &\geq K_{\tan}(t)\left(\dfrac{-K_{\rad}(t)}{K_{\tan}(t)}\right)\\
& = -K_{\rad}(t)
\end{split}
\]
for $-K_{\rad}(t)\geq K_{\tan}(t)$ (indeed, $(1/4)(x^2+4x-1)>x$ for $x>1$), we can see that $c_0(M)\geq \sup_\Sigma\{-K_{\rad}(t)\}$ and conclude by Theorem \ref{Stab-Main} (i) that $\Sigma$ is embedded.
\end{itemize}
 
\end{proof}

\begin{figure}[h]
\includegraphics[scale=0.115]{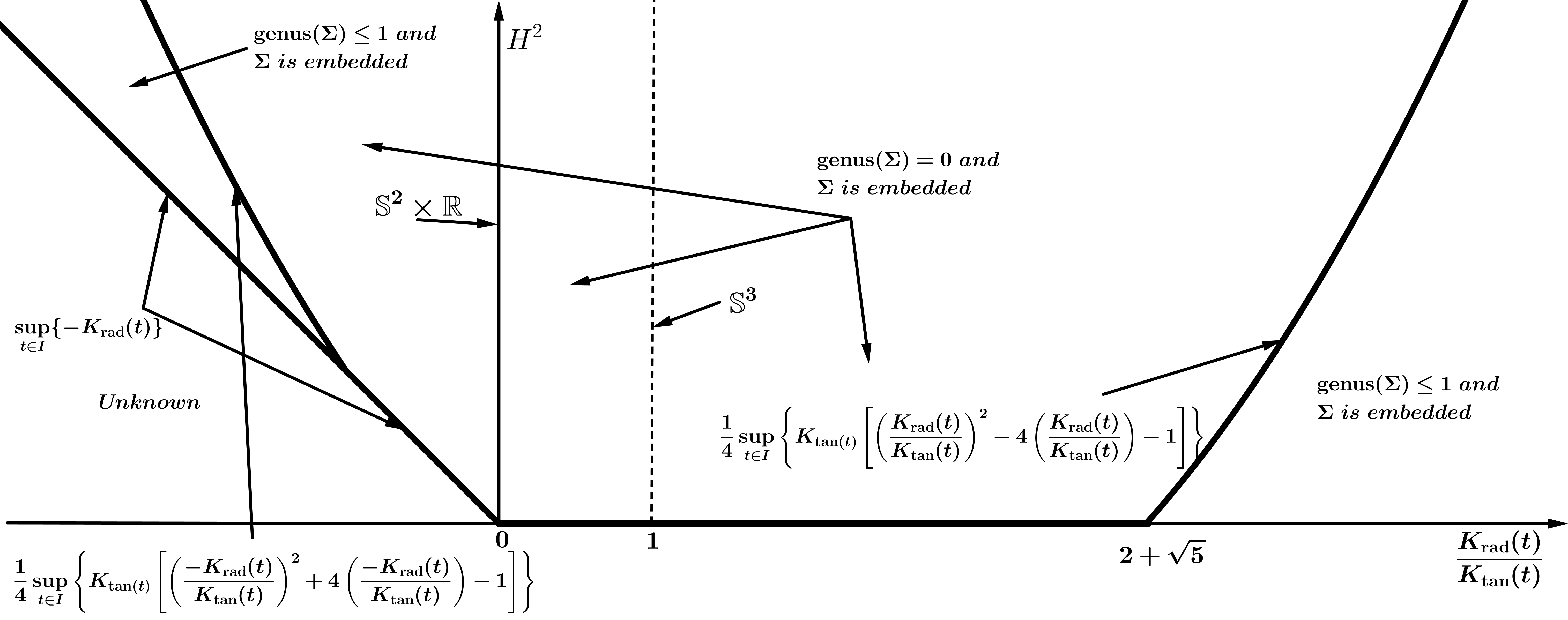}
\caption{Representation of the results in Theorems \ref{Stab-Main}, \ref{theo-warped-1}, and \ref{theo-warped-2}.}
\label{fig-3}
\end{figure}

To conclude this section we give some examples of Riemannian three-dimensional manifolds which satisfies the hypothesis of Theorems \ref{theo-warped-1} and \ref{theo-warped-2}. These examples show that the class of manifolds such that these theorems hold is as large as possible under the assumption that $K_{\tan}(t)>0.$

\begin{example}\label{example}
{\normalfont 
Let $F:I_0\times\R^2\rightarrow\R^4,$ where $I_0\subset\R$ is an interval, be given by 
\[
F(s,\theta)=(s,u(s)\om(\theta))
\] 
be the parametrization of a rotationally symmetric hypersurface of $\R^4,$ where the profile curve is the graphic of the smooth function $u:I_0\rightarrow\R$ and $\om:\R^2\rightarrow\R^3$ is the canonical parametrization of the unit round sphere $\s^2$ in polar coordinates. The three-dimensional Riemannian manifold $M^3=F(I_0\times\R^3)$ has the metric
\[
\langle\cdot,\cdot\rangle = (1+ u'(s)^2)ds^2 + u(s)^2d\om^2.
\]
Define $G:I_0\ria I\subset \R$ by
\[
G'(s)=\sqrt{1+u'(s)^2},\ \mbox{and}\ G(s_0)=t_0,
\]
where $s_0\in I_0$ and $t_0\in I.$ Since $G'(s)>1,$ $G(s)$ is invertible. Let $t=G(s)$ and define $h:I\ria \R$ by 
\[
h(t)=u(G^{-1}(t)).
\] 
With this change of variables, we have
\[
dt=G'(s)ds = \sqrt{1+u'(s)^2}ds \ \mbox{and}\ h(t)=u(s).
\]
Thus $M^3$ can bee seen as a warped product manifold with metric
\[
\lan\cdot,\cdot\ran = dt^2 + h(t)^2d\omega^2.
\]
Since
\[
h'(t)=u'(s)\frac{ds}{dt} = \frac{u'(s)}{\sqrt{1+u'(s)^2}}
\]
and
\[
h''(t)=\frac{d}{ds}\left(\frac{u'(s)}{\sqrt{1+u'(s)^2}}\right)\frac{ds}{dt} = \frac{u''(s)}{(1+u'(s)^2)^2},
\]
the sectional curvatures of these manifolds, in terms of $s,$ are
\[
K_{\tan}(t)=\frac{1-h'(t)^2}{h(t)^2}=\dfrac{1}{u(s)^2(1+u'(s)^2)}
\]
and
\[
K_{\rad}(t)=-\frac{h''(t)}{h(t)} = - \dfrac{u''(s)}{u(s)(1+u'(s)^2)^2}.
\]
For these manifolds, $K_{\tan}(t)>0$ everywhere and the sign of $K_{\rad}(t)$ depends on the sign of $u''(s).$ These manifolds satisfy the hypothesis of Theorems \ref{theo-warped-1} and \ref{theo-warped-2} for every positive smooth function $u:I\rightarrow\R.$ Let see below some particular cases of rotationally symmetric hypersurfaces of $\R^4:$
\begin{itemize}
\item[(i)] The generalized ellipsoids \[E_b^3=\left\{(x,y,z,w)\in\R^4; x^2+y^2+z^2+\frac{w^2}{b^2}=1\right\}, b>0,\] are rotationally symmetric hypersurfaces whose profile curve is the graphic of the function $u:(-b,b)\rightarrow\R$ given by $u(s)=\frac{1}{b}\sqrt{b^2-s^2}.$ The sectional curvatures of these manifolds are
\[
K_{\tan}(t)=\dfrac{b^4}{b^4 + (1-b^2)s^2} \ \mbox{and} \ K_{\rad}(t)=\dfrac{b^6}{(b^4+(1-b^2)s^2)^2}.
\]
Thus, if $b>1,$ then $\dfrac{1}{b^2}\leq \dfrac{K_{\rad}(t)}{K_{\tan}(t)} \leq 1$ and if $0<b<1,$ then $1\leq\dfrac{K_{\rad}(t)}{K_{\tan}(t)}\leq\dfrac{1}{b^2}.$ Therefore the generalized ellipsoids satisfy the hypothesis of Theorem \ref{theo-warped-1} for every $b>\frac{1}{\sqrt{2+\sqrt{5}}}.$ Otherwise, i.e., for $0<b<\frac{1}{\sqrt{2+\sqrt{5}}},$ these manifolds satisfy the hypothesis of Theorem \ref{theo-warped-2}.

\item[(ii)] The generalized hyperboloids \[H_b^3=\left\{(x,y,z,w)\in\R^4; x^2+y^2+z^2-\frac{w^2}{b^2}=1\right\}, b>0,\] are rotationally symmetric hypersurfaces whose profile curve is the graphic of the function $u:\R\rightarrow\R$ given by $u(s)=\frac{1}{b}\sqrt{b^2+s^2}.$ The sectional curvatures of these manifolds are
\[
K_{\tan}(t)=\dfrac{b^4}{b^4 + (1+b^2)s^2} \ \mbox{and} \ K_{\rad}(t)=-\dfrac{b^6}{(b^4+(1+b^2)s^2)^2}.
\]
Thus $-b^2\leq\dfrac{K_{\rad}(t)}{K_{\tan}(t)}<0.$ Therefore, these manifolds satisfy the hypothesis of Theorem \ref{theo-warped-2}.
\end{itemize}
}
\end{example}

%\begin{figure}[h]
%\begin{tabular}{cc}
%\includegraphics[scale=0.44]{stability-2.png}
%&\includegraphics[scale=0.29]{stability-3.png}\\
%$K_{\rad}(t)>0$ & $K_{\rad}(t)<0$\\
%\end{tabular}
%\caption{Representation of the manifolds of Example \ref{example}.}
%\label{fig-2}
%\end{figure}
\section{Proof of Theorem \ref{theo-stab-general}}

We conclude the paper with the proof of Theorem \ref{theo-stab-general}:

\begin{proof}[Proof of Theorem \ref{theo-stab-general}.]

Since $X=\langle X,e_1\rangle e_1 + \langle X,e_2\rangle e_2$ and analogously for $JX,$ we have
\begin{equation}\label{est-aa}
\begin{split}
\sum_{\alpha=4}^n\sum_{i=1}^2\left[\overline{II}^\alpha(e_i,X)^2\right.&\left. + \overline{II}^\alpha(e_i,JX)^2\right]\\
&=\sum_{\alpha=4}^n\left[\sum_{i,j=1}^2\langle X,e_j\rangle^2\overline{II}^\alpha(e_i,e_j)^2 + \langle JX,e_j\rangle^2\overline{II}^\alpha(e_i,e_j)^2\right]\\ 
&=\sum_{\alpha=4}^n\sum_{i,j=1}^2 (\langle X,e_j\rangle^2 + \langle JX,e_j\rangle^2)\overline{II}^\alpha(e_i,e_j)^2\\
&=\sum_{\alpha=4}^n\sum_{i,j=1}^2\overline{II}^\alpha(e_i,e_j)^2\|X\|^2\leq \sum_{\alpha=4}^n \|\overline{II}^\alpha\|^2\|X\|^2,
\\
\end{split}
\end{equation}
provided $X$ and $JX$ is an orthogonal frame of $\Sigma$ and 
\[
\langle X,e_j\rangle^2 + \langle JX,e_j\rangle^2 = \|e_j\|^2\|X\|^2=\|X\|^2.
\]
Since the mean curvature vector of $M^3$ in $\R^n$ is given by
\[
\mathcal{H}=\frac{1}{3}\sum_{\alpha=4}^n (\tr \overline{II}^\alpha)\overline{N}^\alpha,
\]
we have
\[
\|\mathcal{H}\|^2 = \frac{1}{9}\sum_{\alpha=4}^n(\tr \overline{II}^\alpha)^2.
\]
Thus, by using the Gauss equation
\[
\begin{split}
\sum_{\alpha=4}^n\|\overline{II}^\alpha\|^2&=\sum_{\alpha=4}^n(\tr \overline{II}^\alpha)^2-6\scal_M\\
&=9\|\mathcal{H}\|^2 - 6\scal_M\\
\end{split}
\]
and estimate (\ref{est-aa}) in Proposition \ref{main-1}, p. \pageref{main-1}, we have
\[
\begin{split}
Q(X,X) + Q(JX,JX) \geq (4H^2 + 12\scal_M - 9\|\mathcal{H}\|^2)\|X\|^2.
\end{split}
\]
If $\Sigma$ is a compact, without boundary, stable, constant mean curvature $H\neq0$ surface such that $\genus(\Sigma)\geq 1$ then there exist at least two harmonic vector fields $X$ and $JX$ such that, under the hypothesis,
\[
0\leq - \int_\Sigma \left(Q(X,X)+Q(JX,JX)\right) d\Sigma < 0,
\]
which gives a contradiction. Thus, there is no compact, without boundary, stable, constant mean curvature $H\neq0$ surfaces with $\genus(\Sigma)\geq 1$ in $M^3$ under our hypothesis, which implies that $\genus(\Sigma)=0.$
 
\end{proof}

\end{document}